\documentclass[12pt,leqno, final, oneside]{amsart}
\usepackage{showkeys}
\usepackage[latin1]{inputenc}
\usepackage[all]{xy}
\usepackage{bbm}
\usepackage{amsmath}
\usepackage{amssymb}
\usepackage{amsthm}
\usepackage{amsfonts}
\usepackage{graphicx}
\usepackage{eurosym}
\usepackage{color}
\usepackage{nicefrac}
\usepackage[pdftex]{hyperref}
\usepackage{fancyhdr}
\usepackage{pict2e}
\usepackage{lineno}

\newcommand{\IC}{{\mathbb C}}

\newcommand{\IW}{{\mathbb W}}

\newcommand{\IR}{{\mathbb R}}
\newcommand{\IZ}{{\mathbb Z}}

\newcommand{\ka}{{\mathcal A}}

\newcommand{\kh}{{\mathcal H}}
\newcommand{\ki}{{\mathcal I}}
\newcommand{\kj}{{\mathcal J}}
\newcommand{\kk}{{\mathcal K}}
\newcommand{\kl}{{\mathcal L}}

\newcommand{\ko}{{\mathcal O}}
\newcommand{\kp}{{\mathcal P}}

\newcommand{\ku}{{\mathcal U}}

\newcommand{\grass}{\mathcal{G}}

\renewcommand{\labelenumi}{(\roman{enumi})}

\DeclareMathOperator{\dimension}{dim}

\DeclareMathOperator{\Mor}{Mor}

\DeclareMathOperator{\type}{type}
\DeclareMathOperator{\limes}{lim}

\DeclareMathOperator{\conv}{conv}

\newcommand*{\leftindices}[4][]{\ensuremath{%
  {\vphantom{\ifx\\#1\\#2\else#1\fi}}^{#3}_{#4}#2}}

\newtheorem{theorem}{Theorem}[section]
\newtheorem{lemma}[theorem]{Lemma}
\newtheorem{proposition}[theorem]{Proposition}
\newtheorem{corollary}[theorem]{Corollary}

\theoremstyle{remark}
\newtheorem{remark}[theorem]{Remark}

\theoremstyle{definition}
\newtheorem{definition}[theorem]{Definition}
\newtheorem{notation}[theorem]{Notation}

\begin{document}

\title{MV-Polytopes via affine buildings}
\author{Michael Ehrig*}

\thanks{*This research has been partially supported by the EC TMR network "LieGrits", contract MRTNCT 2003-505078 and the DFG-Graduiertenkolleg 1052, 2000 Mathematics Subject Classification. 22E46, 14M15.}

\address{Mathematical Institute, University of Cologne, Weyertal 86-90, 50931 Cologne, Germany}
\email{mehrig@mi.uni-koeln.de}

\begin{abstract}
For an algebraic group $G$, Anderson originally defined the notion of MV-polytopes in \cite{And}, images of MV-cycles, defined in \cite{MirVil}, under the moment map of the corresponding affine Grassmannian. It was shown by Kamnitzer in \cite{Kam1} and \cite{Kam2} that these polytopes can be described via tropical relations and give rise to a crystal structure on the set of MV-cycles. Another crystal structure can be introduced using LS-galleries which were defined by Gaussent and Littelmann in \cite{GauLit}, a more discrete version of Littelmann's path model. 

The main result of this paper is a direct combinatorial construction of the MV-polytopes using LS-galleries. In addition we link this construction to the retractions of the affine building and the Bott-Samelson variety corresponding to $G$. This leads to a definition of MV-polytopes not involving the tropical Pl\"ucker relations. Herefore it provides a description of the polytopes independent of the type of the algebraic group via the gallery model and affine buildings.
\end{abstract}

\maketitle

\tableofcontents


\section{Introduction}
The aim of this paper is to provide a connection between the LS-gallery model \cite{GauLit} for finite-dimensional representations of a connected complex semisimple algebraic group $G$ and the Mirkovi{\'c}-Vilonen polytopes (MV-polytopes for short). These appear in \cite{And} and \cite{Kam1} as images of MV-cycles \cite{MirVil} under the moment map, via retractions in the affine Tits building associated with $G$. It is well known that both, the LS-gallery model and the MV-polytopes are combinatorial realisations of finite-dimensional representations of $G$. We provide a combinatorial link between these two objects and explicitly construct MV-polytopes by starting with an LS-gallery $\delta$ and then defining a new gallery for each vertex of the MV-polytope. In addition we show that the constructed galleries are the images of the retractions of the affine Tits building, when applied to an open subset of the Bialynicki-Birula cell corresponding to $\delta$. 

This gives an alternative definition of MV-polytopes as well as a new proof for Kamnitzer's result that the MV-polytopes are a set of polytopes whose edge lengths are the Lusztig datum of the canonical basis element corresponding to the polytope. This new proof is independent of the type and does not involve the tropical Pl\"ucker relations.

We start by recalling the necessary definitions for the group $G$ (in Section \ref{section_2}), the affine Grassmannian, MV-cycles and MV-polytopes (in Section \ref{section_3}), and LS-galleries and retractions of the Bott-Samelson variety (in Section \ref{section_4}).

Afterwards, in Section \ref{section_5}, we analyse alcoves appearing in a fixed gallery $\delta$ and how one can partition a gallery, according to the properties of these alcoves, into parts consisting of alcoves with the same properties. We also discuss the interaction of the root operators for LS-galleries and our defined partition. In addition we show that by using these partitions one can construct a specific gallery for each element $w$ of the Weyl group of $G$, named $\Xi_w(\delta)$.

In Section \ref{section_6} we prove our main theorem.
\begin{theorem}
Let $M_\delta$ be an MV-cycle corresponding to an LS-gallery $\delta$ and $P=\boldsymbol{\mu}(M_\delta)$ the corresponding MV-polytope, then
$$ P = P_\delta:=\conv(\{wt(\Xi_w(\delta)) \mid w \in W \} ). $$
\end{theorem}

To obtain a proof that this convex hull is equal to the MV-polytope without utilising the tropical Plücker relations, we calculate the images of the retraction for each chamber of the building at infinity in Section \ref{section_6}. We show that for a dense subset of the MV-cycle the image of the retraction at a chamber at infinity, corresponding to a Weyl group element $w$, is equal to our constructed gallery $\Xi_w(\delta)$. This gives the proof for the above mentioned theorem by using the GGMS-stratum (short for Gelfand-Goresky-MacPherson-Serganova stratum) and results from \cite{GauLit}.

\subsection*{Acknowledgements.} The author would like to thank Pierre Baumann, Stéphane Gaussent, Peter Littelmann, Cristoph Schwer, and Catharina Stroppel for various helpful discussions, remarks, and ongoing good advice. The author would also like to thank the European Research Training Network "LieGrits" for the support and the "Institut de Mathématiques Élie Cartan" at the university Nancy I, Henri Poincaré and especially Stéphane Gaussent for the hospitality during the summer term 2007.


\section{Algebraic and Kac-Moody groups}\label{section_2}
We begin with fixing the notations for our group $G$ and the associated affine Kac-Moody group $\hat{\kl}(G)$. We also recall a number of technical results for calculations in these groups from \cite{Tits1} and \cite[§6]{Stein}.

\subsection*{Notations for the group $G$}

Let $G$ be a complex, simply-connected, semisimple algebraic group. We fix a \textit{Borel subgroup} $B \subset G$ and a \textit{maximal torus} $T \subset B$ and denote by $B^-$ the opposite Borel subgroup of $B$ relative to $T$. We denote the \textit{unipotent radicals} of $B$ and $B^-$ by $U$ and $U^-$. Finally let $W$ be the Weyl group of $G$.

We denote by $X=X^*(T):=\Mor(T,\IC^*)$, respectively $X^\vee=X_*(T):=\Mor(\IC^*,T)$, its \textit{character}, respectively \textit{cocharacter group}. For $\mu \in X^\vee$, we write $\mu:\IC^* \rightarrow T$, $s \mapsto s^\mu$ to simplify the notations in the calculations. Furthermore we write $\Phi$ and $\Phi^\vee=\{\alpha^\vee \mid \alpha \in \Phi \}$ for the \textit{root} and \textit{coroot system}. Corresponding to our choice of $B$ we denote by $\Phi^+$ and $\Phi^-$ the positive and negative roots of $G$ and use the notation $\Phi^\vee_+$ and $\Phi^\vee_-$ for the corresponding subsets of the coroots. By $X_+=\{\lambda \in X \mid \forall \alpha^\vee \in \Phi^\vee_+, \left\langle \lambda, \alpha^\vee \right\rangle \geq 0 \}$ and $X^\vee_+=\{\lambda^\vee \in X^\vee \mid \forall \alpha \in \Phi_+, \left\langle \alpha, \lambda^\vee \right\rangle \geq 0 \}$ we denote the sets of \textit{dominant weights} and \textit{coweights}. 

We fix a numbering of the simple roots $(\alpha_i)_{i \in I}$. Inside $X^+$ we denote by $\Lambda_i$ the \textit{fundamental weight} corresponding to $\alpha_i$ for each $i \in I$. 
We denote by $s_i$ the generator of the Weyl group that acts on $X$ as the reflection along $\alpha_i$. By $l(w)$ for an element $w \in W$ we denote the length of $w$. On $X$, respectively $X^\vee$ we denote by "$\geq$" the usual dominance order.

Following, we fix certain elements inside $G$ and list their commutator relations (see also \cite{Tits1}). For every simple root $\alpha$, we denote by $U_{\alpha}=\{x_{\alpha}(t) \mid t \in \IC \}$ a non-trivial additive 1-parameter subgroup of $U$ such that $s^{\lambda^\vee} x_{\alpha}(t) s^{-\lambda^\vee} = x_{\alpha}(s^{\left\langle \alpha, \lambda^\vee \right\rangle}t)$ holds for all $\lambda^\vee \in X^\vee$, $s \in \IC^*$, and $t \in \IC$. By general theory we know that, for each $i \in I$ there exists a unique morphism $\phi_i:{\rm SL}_2 \rightarrow G$ such that
$$\phi_i \left(
\begin{array}{cc}
1 & t \\
0 & 1
\end{array} \right) = x_{\alpha_i}(t) \text{ and } \phi_i \left(
\begin{array}{cc}
s & 0 \\
0 & s^{-1}
\end{array} \right) = s^{\alpha_i^\vee} $$ for all $s \in \IC^*$ and $t \in \IC$. Furthermore we fix the elements
$$x_{-\alpha_i}(t) = \phi_i \left(
\begin{array}{cc}
1 & 0 \\
t & 1
\end{array} \right) \text{ and } \overline{s_i} = \phi_i \left(
\begin{array}{cc}
0 & 1 \\
-1 & 0
\end{array} \right)$$

For an arbitrary element $w \in W$, we write $\overline{w}$ for the lift in $N_G(T)$ by using the elements $\overline{s_i}$ instead of the $s_i$'s in any reduced expression for $w$, which is well defined as the $\overline{s_i}$ satisfy the braid relations.

For an arbitrary positive root $\alpha$, we choose a simple root $\alpha_i$ and an element $w \in W$ such that $\alpha=w\alpha_i$. We define analogous one-parameter subgroups $U_\alpha$ and $U_{-\alpha}$ for $\alpha$ by
$$ x_\alpha(t):=\overline{w}x_{\alpha_i}(t)\overline{w}^{-1} \text{ and } x_{-\alpha}(t):=\overline{w}x_{-\alpha_i}(t)\overline{w}^{-1} $$
and write $\overline{s_\alpha}$ for the element $\overline{w} \overline{s_i} \overline{w}^{-1}$. For these elements we want to list a number of relations that are satisfied (see also \cite[§3.6]{Tits1} or \cite[§6]{Stein}):
\begin{enumerate}
\item For all $\lambda^\vee \in X^\vee$, a root $\alpha$, $s \in \IC^*$, and $t \in \IC$,
$$ s^\lambda x_\alpha(t) = x_\alpha(s^{\left\langle \alpha, \lambda^\vee \right\rangle}t)s^\lambda. $$
\item For $\alpha \in \Phi$ and $t,t' \in \IC$ such that $1+tt' \neq 0$,
$$x_\alpha(t)x_{-\alpha}(t')=x_{-\alpha}(t'/(1+tt'))(1+tt')^{\alpha^\vee}x_\alpha(t/(1+tt')).$$
\item For a positive root $\alpha$ and $t \in \IC^*$,
$$ x_\alpha(t)x_{-\alpha}(-t^{-1})x_\alpha(t)=x_{-\alpha}(-t^{-1})x_\alpha(t)x_{-\alpha}(-t^{-1})=t^{\alpha^\vee}\overline{s}_\alpha=\overline{s}_\alpha t^{-\alpha^\vee}.$$
\item (Chevalley's commutator formula) If $\alpha$ and $\beta$ are two linearly independent roots, then there are numbers $c_{i,j,\alpha,\beta} \in \{\pm 1,\pm 2,\pm 3 \}$ such that
$$ x_\beta(s)^{-1}x_\alpha(t)^{-1}x_\beta(s)x_\alpha(t) = \prod_{i,j >0}^{\rightarrow} x_{i\alpha + j\beta}(c_{i,j,\alpha,\beta}t^i s^j) $$
for all $s,t \in \IC$. The product is taken over all pairs $i,j \in \IZ^+$ such that $i\alpha + j\beta$ is a root and in order of increasing height of the occurring roots, i.e., with increasing products $\left\langle i\alpha + j\beta, \rho^\vee\right\rangle$, with $\rho^\vee$ being the half-sum over the positive coroots.
\end{enumerate}

Especially the third relation is very important in the later calculations as it will pretty much correspond to the folding of a gallery at a certain face. While the second relation will be important in some independence arguments in the proof of the main result in Section \ref{section_6}.

\begin{remark}
Whenever there is no risk of confusion, we write write $w$, respectively $s_i$, instead of $\overline{w}$, respectively $\overline{s_i}$, for representatives of the Weyl group elements in $N_G(T)$.
\end{remark}

We write $\ko$ for the ring of formal power series $\IC [[t]]$ and $\kk$ for its field of fractions $\IC ((t))$. We denote by $G(\ko)$ and $G(\kk)$, the sets of $\ko$-valued and $\kk$-valued points of $G$. We will not explicitly use the affine Kac-Moody group $\hat{\kl}(G)$ and hence refer to \cite[§13]{Kum} or \cite{GauLit} for the definition and its relation with the Bott-Samelson variety.
We denote by $W^\mathfrak{a} = W \ltimes X^\vee$ the affine Weyl group.

We denote by $ev : G(\ko) \rightarrow G$ and $ev_\kl:\kl(G(\ko)) \rightarrow \IC^* \times G$ the evaluation maps at $t=0$. We define the corresponding \textit{Iwahori subgroup} as the preimages of the Borel subgroup $\ki=ev^{-1}(B)$.

It is known that by this construction of $\hat{\kl}(G)$ we obtain a new simple root $\alpha_0$ and a corresponding reflection $s_0 \in W^\mathfrak{a}$. All notations and computation rules in the previous part apply to this root as well. Although we will not define any operators for these affine roots, they still occur in the description of elements in the Bott-Samelson variety as we will see later.

As in \cite[§5.1]{BauGau} we will identify the real affine roots with $\Phi \times \IZ$, where we identify $\alpha \in \Phi$ with $(\alpha,0)$ and $\alpha_0$ with $(-\theta,-1)$, where $\theta$ is the highest root in the classical root system. To each pair $(\alpha,n)$ we associate a reflection of $X^\vee \otimes \IR$ as $s_{\alpha,n}(x):=x-(\left\langle \alpha, x \right\rangle - n) \alpha^\vee$.

Via these reflections we can define an action of the affine Weyl group on $X^\vee \otimes \IR$. This allows us to define an action of it on the set of real affine roots, via $(\tau_\lambda,w)(\alpha,n):=(w\alpha, n + \left\langle \alpha, \lambda \right\rangle)$, for $\tau_\lambda$ the translation with respect to an element $\lambda \in X^\vee$ and $w \in W$. These two actions are compatible with each other, meaning that if we apply an element $w$ of the affine Weyl group to a reflection hyperplane $H_{\alpha,n}$ we obtain the reflection hyperplane corresponding to $w.(\alpha,n)$.

We denote the element $\alpha_0 - \theta$ by $\underline{\delta}$. In the above notations this means that $\underline{\delta}$ corresponds to $(0,-1)$ and thus any real affine root can be written, in a unique way, in the form $\alpha + n \underline{\delta}$, for $\alpha \in \Phi$ and $n \in \IZ$.

To the root $(\alpha,n)$ (or $\alpha - n\underline{\delta}$) we associate the one-parameter subgroup 
$$ U_{\alpha,n}= \{ x_{\alpha}(at^n) \mid a \in \IC \}. $$


\section{Grassmannian and MV-cycles}\label{section_3}
\subsection*{The affine Grassmannian}

We denote by $\grass_G:=G(\kk)/G(\ko)$ the \textit{affine Grassmannian} associated to $G$. If no confusion can arise we will usually denote $\grass_G$ by $\grass$.

Inside the affine Grassmannian we are interested in two different types of orbits. Obviously $G(\ko)$ acts on $\grass$ by left multiplication, it does so with finite dimensional orbits, and one can easily index these orbits. Let $\lambda \in X^\vee_+$, by definition we can view $\lambda$ as an element in $G(\kk)$ and denote its image in $\grass$ by $L_\lambda$. We denote by $\grass_\lambda=G(\ko).L_\lambda$ the corresponding $G(\ko)$-orbit. It suffices to use $\lambda \in X^\vee_+$ as the orbit is closed under the Weyl group action. The closure $\overline{\grass_\lambda}$ is called a \textit{generalized Schubert variety}.
 
The second type of orbit we are interested in are the semi-infinite orbits $S_\nu^w$ for $\nu \in X^\vee$ and $w \in W$. These are defined as $S_\nu^w=wU^-(\kk)w^{-1}. L_\nu$ and we usually denote $S_\nu^{\rm id}$ by $S_\nu$.

For both types of orbits there are well known closure relations, see \cite{MirVil}.

\subsection*{MV-cycles and MV-polytopes}

We are interested in the intersections of the above described orbits. The next definition is based on \cite{MirVil}.

\begin{definition}\textbf{(\cite[§5.3, Def. 2]{And} and \cite[2.2]{Kam2})}
Let $\lambda \in X^\vee_+$ and $\mu \in X^\vee$. If the intersection $\grass_\lambda \cap S_\mu$ is not empty we call the irreducible components of $\overline{\grass_\lambda \cap S_\mu}$ the \textit{MV-cycles of coweight} $(\lambda,\mu)$.
\end{definition}

By work of Mirkovi{\'c} and Vilonen, it is known that the collection of all MV-cycles of coweights $(\lambda,\nu)$ for $\nu \in X^\vee$, form a natural basis of the irreducible representation $V(\lambda)$ for $G^\vee$ of highest weight $\lambda$, where $G^\vee$ is the Langlands dual group of $G$. This is done in such a way that the MV-cycles of coweight $(\lambda,\nu)$ span the weight space for the coweight $\nu$.

\begin{definition}\textbf{(\cite[§6, Prop 4]{And} and \cite[2.2]{Kam2})}
Let $A$ be an MV-cycle of coweight $(\lambda,\mu)$ then we define its corresponding \textit{MV-polytope} $P(A)$ as the convex hull of the set $\{\mu \in X^\vee \mid L_\mu \in A \}$ inside $X^\vee \otimes \IR$. We call polytopes in $X^\vee \otimes \IR$ arising in this way \textit{MV-polytopes}.
\end{definition}

As each MV-cycle is a $T$-invariant closed subvariety of the affine Grassmannian, it was shown by Anderson (\cite{And}) that this convex hull is the image of $A$ under the moment map $\boldsymbol{\mu}$ of the affine Grassmannian with respect to the $T$-action. This is done by embedding the affine Grassmannian into an infinite dimensional projective space $\mathbb{P}(V)$ with $V$ being the highest weight representation of the affine Kac Moody group $\hat{\kl}(G)$ with highest weight $\Lambda_0$, the fundamental weight corresponding to $\alpha_0$. This representation decomposes into weight spaces and thus we write $v=\sum_{\nu \in X}v_\nu$, for $v \in V$ and define
$$ \boldsymbol{\mu}([v]) = \sum_{\nu \in X} \frac{| v_\nu|^2}{|v|^2}\nu,$$ 
see \cite[§6]{And}.

In addition to the images of the MV-cycles under the moment map, one also needs to know the images of the semi-infinite cells. Due to their closure relations, \cite{MirVil}, it follows that
$$\boldsymbol{\mu}(\overline{S_\mu^w})=C_\mu^w:=\{p \in X^\vee \otimes \IR \mid \left\langle p, w \cdot \Lambda_i \right\rangle \geq \left\langle \mu, w \cdot \Lambda_i \right\rangle \text{ for all }i \},$$
where the $\Lambda_i$'s are the fundamental coweights, \cite[2.2]{Kam2}.

It was shown by Kamnitzer that MV-polytopes can be written as the intersection of cones of the form $C_{\mu_w}^w$ for $w \in W$. Since the MV-polytopes are the images of MV-cycles under the moment map, one can try to lift this description into the Grassmannian. This leads to the following description via GGMS-strata.

\begin{definition}\textbf{(\cite[§2.4]{Kam2})}
Let $\mu_\bullet =(\mu_w)_{w \in W}$ be a set of coweights, such that $\mu_v \geq_w \mu_w$ for all $v$, $w \in W$. The corresponding \textit{Gelfand-Goresky-MacPherson-Serganova} (or short GGMS) strata on the affine Grassmannian is
$$ A(\mu_\bullet) := \bigcap_{w \in W} S^w_{\mu_w}. $$
\end{definition}

The assumed inequalities for the coweights are exactly those that one needs to demand for the intersection to be non-empty. We will later go into more detail, why the GGMS-strata is of such great importance for the gallery model. It is a quite easy calculation to show that indeed this definition produces the right variety and that one has
$$ \boldsymbol{\mu}(\overline{A(\mu_\bullet)})=P(\mu_\bullet), $$
for a set of coweights satisifying the needed inequalities, see \cite[2.5]{Kam2}.


\section{Galleries}\label{section_4}
Most of the following definitions and constructions follow \cite{GauLit} with a small amount of modifications at some places and simplifications at others. In addition we also want to recall the definition of the Bott-Samelson variety as it is used in both \cite{Kum} and \cite{GauLit}. For some of the original statements, we refer to \cite{Dem} and \cite{Han}.

\subsection*{Combinatorial galleries}

We will now sum up and recall many of the definition from \cite{GauLit} concerning combinatorial galleries. Starting with the most basic sets.

\begin{definition}
We denote by $\ka:=X^\vee \otimes_\IZ \IR$ the real span of the coweight lattice, called the \textit{fundamental apartment}, and by $H^\mathfrak{a}= \{ H_{\alpha,m} \mid \alpha \in \Phi^+, m \in \IZ \} \subset \ka$ the set of affine reflection hyperplanes for the action of $W^\mathfrak{a}$ on $\ka$, where \break $H_{\alpha,m}=\{x \in \ka \mid \left\langle x,\alpha \right\rangle = m \}$. The corresponding reflection for the hyperplane $H_{\alpha,m}$ will be denoted by $s_{\alpha,m}$ and by $H_{\alpha,m}^+$ and $H_{\alpha,m}^-$ we denote the two corresponding (closed) half-spaces.

A connected component of $\ka \setminus H^\mathfrak{a}$ is called \textit{open alcove} and its closure \textit{(closed) alcove}.

The \textit{fundamental alcove} is the subset $\Delta_f:=\{x \in \ka \mid 0 \leq \left\langle x, \beta \right\rangle \leq 1 \ \forall \beta \in \Phi^+ \}$.

A \textit{face} $F$ is a subset of $\ka$ of the form
$$ \bigcap_{(\beta,m) \in \Phi^+ \times \IZ} H_{\beta,m}^{\epsilon_{\beta,m}}, $$
with $\epsilon_{\beta,m} \in \{+,-,\emptyset \}$ and $H_{\beta,m}^\emptyset = H_{\beta,m}$.

We define $S^\mathfrak{a}=\{s_{\beta,m} \mid H_{\beta,m} \text{ contains a face of } \Delta_f \}$ the set of affine reflections that generates the affine Weyl group $W^\mathfrak{a}$.

A connected component of $\ka \setminus \bigcup_{\beta \in \Phi^+} H_{\beta,0}$ is called an \textit{open chamber} and its closure is called a \textit{(closed) chamber}.

We denote by $\mathfrak{C}_f$ the \textit{dominant chamber}, which is the chamber that contains $\Delta_f$ and by $\mathfrak{C}_{-f}$ the \textit{anti-dominant chamber}, which is $w_0\mathfrak{C}_f$. All chambers are in the $W$-orbit of $\mathfrak{C}_f$. 

We say that an alcove $\Delta$ can be seperated from a chamber $\mathfrak C$ by a reflection hyperplane $H \in H^\mathfrak a$, if there exists a translation $\tau \in W^\mathfrak a$ such that $\Delta$ and $\tau \mathfrak C$ lie on different sides of the hyperplane $H$. 
\end{definition}

An important notion will be the type of a face.

\begin{definition}
For a face $F$ of $\Delta_f$ we define 
$$ S^\mathfrak{a}(F)=\{s_{\beta,m} \in S^\mathfrak{a} \mid F \subset H_{\beta,m} \}$$
and call this the \textit{type} of $F$. This means that 
$$S^\mathfrak{a}(0)=S=\{s_{\beta,0} \mid H_{\beta,0} \text{ is a wall for } \Delta_f \}$$
and $S^\mathfrak{a} (\Delta_f) = \emptyset$. One defines the type of an arbitrary face $F'$ as the type of the unique face of $\Delta_f$ that lies in the $W^\mathfrak{a}$-orbit of $F'$. Hence the type of an alcove if always trivial, while the type of a codimension 1 face always consists of exactly one element of $S^\mathfrak a$.
\end{definition}

Using the defined sets in $\ka$ and the type we can define the notion of a gallery.

\begin{definition}
A \textit{combinatorial gallery joining $0$ with $\mu$} in $\ka$ is a sequence
$$\gamma = ( F_f = \Gamma_0' \subset \Gamma_0 \supset \Gamma_1' \subset \ldots \supset \Gamma_p' \subset \Gamma_p \supset F_\mu ),$$
such that 
\begin{itemize}
\item $F_f$ is the face of type $S$ in $\Delta_f$,
\item $F_\lambda$ is the face corresponding to $\mu$, via the embedding $X^\vee \rightarrow \ka$,
\item the $\Gamma_j$'s are alcoves,
\item each $\Gamma_j'$, for $j \in \{1, \ldots , p \}$, is a face of $\Gamma_{j-1}$ and $\Gamma_j$, of relative dimension one.
\end{itemize}

Let $\mu \in X^\vee$ and $\gamma$ be a combinatorial gallery joining $0$ and $\mu$
$$\gamma = ( F_f=\Gamma_0' \subset \Gamma_0 \supset \Gamma_1' \subset \ldots \supset \Gamma_p' \subset \Gamma_p \supset F_\mu ).$$
Then the \textit{type} of $\gamma$ is the list of the types of all the small faces of the gallery $\gamma$:
$\type(\gamma)=(t_1, \ldots, t_p)$, 
where $t_j$ is the type of the face $\Gamma_j'$ for all $0 \leq j \leq p$.

Let $\lambda \in X^\vee_+$ and a minimal gallery $\gamma_\lambda$ joining $0$ with $\lambda$, i.e., a gallery such that there exists no gallery consisting of fewer alcoves connecting $0$ and $\lambda$. Then we denote by $\Gamma(\gamma_\lambda)$ the set of all combinatorial galleries starting at the origin, with the same type as $\type(\gamma_\lambda)$.
\end{definition}

\begin{remark}
If the type of a combinatorial gallery $\gamma$ is fixed, we also use the notation $\gamma = [\gamma_0,\ldots, \gamma_p]$ with $\gamma_i \in t_i \cup \{id\}$ and $\Gamma_i = \gamma_0 \cdots \gamma_i \Delta_f$.
\end{remark}

Since we are interested in a certain type of gallery, we need the following notions for the faces.

\begin{definition}
Let $\delta = ( F_f = \Delta_0' \subset \Delta_0 \supset \Delta_1' \subset \ldots \supset \Delta_p' \subset \Delta_p \supset F_\mu)$ be in $\Gamma^+(\gamma_\lambda)$. For $j=0,\ldots,p$ let $\kh_j$ be the set of all affine reflection hyperplanes $H$ in $H^\mathfrak a$ such that $\Delta_j' \subset H$. In all cases except $j=0$ this set will consist of a single hyperplane and for $j=0$ it will consist of all classical reflection hyperplanes. We say that an affine hyperplane $H$ is a \textit{load-bearing} wall for $\delta$ at $\Delta_j$ if $H \in \kh_j$ and $H$ separates $\Delta_j$ from $\mathfrak C_{-f}$.

The set of all load-bearing walls of a combinatorial gallery $\delta$, is denoted by $J_{-\infty}(\delta)$. This set can be divided into two subsets $J_{-\infty}^+(\delta)$ and $J_{-\infty}^-(\delta)$. An index $j \in J_{-\infty}(\delta)$ is in $J_{-\infty}^+(\delta)$ if $\Gamma_{j-1}$ is not separated from $\mathfrak C_{-f}$, otherwise $j \in J_{-\infty}^-(\delta)$.

For $\alpha \in \Phi$ simple, define $\IW_\alpha(\delta) = \{ j \in \{0, \ldots, p \} \mid \Gamma_j' \subset H_{\alpha,n} \text{ for some } n \in \IZ \}$ (or $\IW_\alpha$ if no confusion can arise).
\end{definition}

There are a number of different subsets of combinatorial galleries that will be important in the succeeding parts, the first of these are the positively or negatively folded galleries.

\begin{definition} \label{positive_gallery}
The gallery $\delta \in \Gamma(\gamma_\lambda)$ is called \textit{positively folded} at $\Delta_j'$ if $\Delta_{j-1} = \Delta_j$ and the half-space, which contains $\Delta_j$ and corresponds to the reflection hyperplane containing $\Delta_j'$, can be separated from the anti-dominant chamber $\mathfrak C_{-f}$.

We say that the gallery is positively folded if all foldings are positive.

The set of positively folded galleries of type $\type(\gamma_\lambda)$ will be denoted by $\Gamma^+(\gamma_\lambda)$.

A \textit{negative folding} is defined in the analogous way with $\mathfrak C_f$ instead of $\mathfrak C_{-f}$. The definitions of a negatively folded gallery and the set $\Gamma^-(\gamma_\lambda)$ are parallel as well.
\end{definition}

\begin{remark}
If a gallery $\delta$ is positively folded at $\Delta_j' \subset H$ then $H$ is by definition a load-bearing wall and in $J_{-\infty}^-(\delta)$.
\end{remark}

For our purposes we not only need positively folded galleries, but also those of maximal dimension.

\begin{definition} \textbf{(\cite[§5.2]{BauGau} and \cite[§5, Def. 15]{GauLit})} \label{LS_gallery}
The \textit{dimension of a gallery} $\gamma \in \Gamma^+(\gamma_\lambda)$ is defined as:
$${\rm dim} \delta = \# \{ (H,\Delta_j) \mid H \text{ is a load-bearing wall for $\delta$ at $\Delta_j$} \}.$$

A positively folded gallery $\delta$ of type $t_(\gamma_\lambda)$ joining the origin with $\nu$ is called an \textit{LS-gallery of type $t_{\gamma_\lambda}$} if ${\rm dim} \gamma = \left\langle \lambda + \nu,\rho \right\rangle + {\rm dim}(P_\lambda / B)$, where $P_\lambda$ is the standard parabolic subgroup $P_J$ for $J=\{j \in I \mid \left\langle \alpha_j, \lambda \right\rangle =0 \}$. The set of all LS galleries joining the origin and $\nu$ of type $t_{\gamma_\lambda}$ is denoted by $\Gamma^+_{LS}(\gamma_\lambda)$.
\end{definition}

\begin{remark}
By \cite{GauLit}, LS-galleries are the galleries of maximal dimension, as this is always bounded by $\left\langle \lambda + \nu,\rho \right\rangle$.
\end{remark}

Replacing $\mathfrak C_{-f}$ by $w \mathfrak C_{-f}$ in the definitions \ref{positive_gallery} and \ref{LS_gallery}, we obtain:

$\Gamma^w(\gamma_\lambda)$: combinatorial galleries of type $\type(\gamma_\lambda)$, positively folded with respect to $w \mathfrak C_{-f}$.

$\Gamma^w_{LS}(\gamma_\lambda)$: combinatorial galleries of type $\type(\gamma_\lambda)$, LS with respect to $w \mathfrak C_{-f}$.

$\Gamma^w(\gamma_\lambda,\nu)$: combinatorial galleries of type $\type(\gamma_\lambda)$, positively folded with respect to $w \mathfrak C_{-f}$ joining the origin and $\nu$.

$\Gamma^w_{LS}(\gamma_\lambda,\nu)$: combinatorial galleries of type $\type(\gamma_\lambda)$, LS with respect to $w \mathfrak C_{-f}$ joining the origin and $\nu$.

\subsection*{Root operators}

For the general notions of a crystal we refer to \cite{Kash2}. Most of the definitions we give can be modified for the path model and most of the combinatorics we deal with, in Section $\ref{section_5}$, work with path as well. For more details about the path-model we refer to \cite{Lit3}, \cite{Lit4}, \cite{Lit6}, \cite{Lit2}, and \cite{Lit5}.

Let $\alpha$ be a simple root, $\lambda$ a dominant coweight, and $\nu \prec \lambda$ an arbitrary coweight. We fix $\delta \in \Gamma^+_{LS}(\gamma_\lambda, \nu)$ for a given type $\gamma_\lambda$ as
$$\delta = ( F_f=\Delta_0' \subset \Delta_0 \supset \Delta_1' \subset \ldots \supset \Delta_p' \subset \Delta_p \supset F_\mu ) = [\delta_0, \ldots, \delta_p].$$

\renewcommand{\labelenumi}{\Roman{enumi})}

\begin{definition} \textbf{(\cite[§6, Def. 16]{GauLit})}

Let $\delta \in \Gamma^+_{LS}(\gamma_\lambda)$ and ${\boldsymbol{\mathfrak m}_\alpha^\delta}$ be minimal with the property that a face $\Delta_k'$ is contained in $H_{\alpha,{\boldsymbol{\mathfrak m}_\alpha^\delta}}$.
\begin{enumerate}
\item If ${\boldsymbol{\mathfrak m}_\alpha^\delta} \leq -1$. Let ${\boldsymbol{\mathfrak t}_\alpha^\delta}$ be minimal such that $\Delta_{\boldsymbol{\mathfrak t}_\alpha^\delta}' \subset H_{\alpha,{\boldsymbol{\mathfrak m}_\alpha^\delta}}$ and fix $0 \leq {\boldsymbol{\mathfrak s}_\alpha^\delta} \leq {\boldsymbol{\mathfrak t}_\alpha^\delta}$ maximal such that $\Delta_{\boldsymbol{\mathfrak s}_\alpha^\delta}' \subset H_{\alpha,{\boldsymbol{\mathfrak m}_\alpha^\delta}+1}$.

Then $e_\alpha \delta$ is defined by
$$e_\alpha \delta = ( F_f=\widetilde{\Delta_0'} \subset \widetilde{\Delta_0} \supset \widetilde{\Delta_1'} \subset \ldots \supset \widetilde{\Delta_p'} \subset \widetilde{\Delta_p} \supset F_{\mu + \alpha^\vee} ),$$
where $\widetilde{\Delta_i} = \left\{ \begin{array}{ll} \Delta_i & \text{for } i < {\boldsymbol{\mathfrak s}_\alpha^\delta} \\ s_{\alpha,{\boldsymbol{\mathfrak m}_\alpha^\delta}+1}(\Delta_i) & \text{for } {\boldsymbol{\mathfrak s}_\alpha^\delta} \leq i < {\boldsymbol{\mathfrak t}_\alpha^\delta} \\ t_{\alpha^\vee}(\Delta_i) & \text{for } i \geq {\boldsymbol{\mathfrak t}_\alpha^\delta}. \end{array} \right. $ \\
and $t_{\alpha^\vee}$ is the translation by $\alpha^\vee$.

\item If $\left\langle \nu, \alpha \right\rangle - {\boldsymbol{\mathfrak m}_\alpha^\delta} \geq 1$. Let ${\boldsymbol{\mathfrak j}_\alpha^\delta}$ be maximal such that $\Delta_{\boldsymbol{\mathfrak j}_\alpha^\delta}' \subset H_{\alpha,{\boldsymbol{\mathfrak m}_\alpha^\delta}}$ and fix ${\boldsymbol{\mathfrak j}_\alpha^\delta} \leq {\boldsymbol{\mathfrak k}_\alpha^\delta} \leq p+1$ minimal such that $\Delta_{\boldsymbol{\mathfrak k}_\alpha^\delta}' \subset H_{\alpha,{\boldsymbol{\mathfrak m}_\alpha^\delta}+1}$.

Then $f_\alpha \delta$ is defined by
$$f_\alpha \delta = ( F_f=\widetilde{\Delta_0'} \subset \widetilde{\Delta_0} \supset \widetilde{\Delta_1'} \subset \ldots \supset \widetilde{\Delta_p'} \subset \widetilde{\Delta_p} \supset F_{\mu - \alpha^\vee} ),$$
where 
$\widetilde{\Delta_i} = \left\{ \begin{array}{ll} \Delta_i & \text{for } i < {\boldsymbol{\mathfrak j}_\alpha^\delta} \\ 
s_{\alpha,{\boldsymbol{\mathfrak m}_\alpha^\delta}}(\Delta_i) & \text{for } {\boldsymbol{\mathfrak j}_\alpha^\delta} \leq i \leq {\boldsymbol{\mathfrak k}_\alpha^\delta}-1 \\
t_{-\alpha^\vee}(\Delta_i) & \text{for } i \geq {\boldsymbol{\mathfrak k}_\alpha^\delta}. \end{array} \right.$ .
\end{enumerate}
\end{definition}

\renewcommand{\labelenumi}{(\roman{enumi})}

\begin{remark}
If the operators $e_\alpha$ or $f_\alpha$ are defined, their images are again LS-galleries, \cite{GauLit}.
\end{remark}

\begin{notation}
We will use the notation ${\boldsymbol{\mathfrak m}_\alpha^\gamma}$, ${\boldsymbol{\mathfrak j}_\alpha^\gamma}$, ${\boldsymbol{\mathfrak k}_\alpha^\gamma}$, ${\boldsymbol{\mathfrak s}_\alpha^\gamma}$, and ${\boldsymbol{\mathfrak t}_\alpha^\gamma}$ throughout the proceeding sections.
\end{notation}

\begin{definition}
We fix the following
\begin{itemize}
\item $wt(\delta) = \mu$,
\item $\epsilon_\alpha(\delta) = {\rm max}_m \{ e_\alpha^m(\delta) \text{ is defined} \}$,
\item $\varphi_\alpha(\delta) = {\rm max}_m \{ f_\alpha^m(\delta) \text{ is defined} \}$,
\item $e_\alpha^{\max}(\delta) = e_\alpha^{\varepsilon_\alpha(\delta)}(\delta),$
\item $f_\alpha^{\max}(\delta) = f_\alpha^{\varphi_\alpha(\delta)}(\delta).$
\end{itemize}
\end{definition}

For more details about the properties of these root operators, we refer to \cite{GauLit}, especially §6, Corollary 1, Lemma 5, and Theorem 2.

\subsection*{Bott-Samelson resolution}

The two facts that on the one hand we have restricted ourselves to galleries of alcoves and on the other hand that the Kac-Moody group is affine simplify our definition of the Bott-Samelson variety, see \cite[§ 5.2]{BauGau}.

\begin{definition}\textbf{(\cite[§3.3, Example 3]{GauLit})}
Let $F$ be a face of the fundamental alcove and let us denote by $W^\mathfrak{a}(F)$ the subgroup of $W^\mathfrak{a}$ generated by its type $S^\mathfrak{a}(F)$, then we define the \textit{standard parahoric subgroup of type $F$} as
$$P_F := \bigcup_{w \in W^\mathfrak{a}(F)} \ki w \ki. $$
\end{definition}

\begin{definition}\textbf{(\cite[§5.2]{BauGau})}
For $0 \leq j \leq p$, let us denote by $\kp_j$ the parahoric subgroup of $G(\kk)$ of type $t_j$, containing $\ki$. The \textit{Bott-Samelson variety} $\hat{\Sigma}(\gamma_\lambda)$ is defined as
$$ \hat{\Sigma}(\gamma_\lambda) = G(\ko) \times_{\ki} \kp_1 \times_{\ki} \ldots \times_{\ki} \kp_p / \ki. $$
In other words it is defined as the quotient of the group $G(\ko) \times \kp_1 \times \ldots \times \kp_p$ by the subgroup $\ki \times \ki \times \ldots \times \ki$, the $p+1$st power of $\ki$ under the right action given by $g.q=(g_0 q_0,q_0^{-1}q_1q_1,\ldots,q_{p-1}^{-1}g_pq_p)$ where $q=(q_0,\ldots,q_p) \in G(\ko) \times \kp_1 \times \ldots \times \kp_p$ and $g=(g_1,\ldots,g_p) \in \ki \times \ki \times \ldots \times \ki$.
\end{definition}

\begin{remark}\label{BS_gall}
By a result of Contou-Carrère, \cite[I, §6]{Con}, and a generalization of it, \cite[§7, Def.-Prop. 1]{GauLit}, one can view the Bott-Samelson variety as the set of galleries of type $\gamma_\lambda$ in the affine building $\kj^\mathfrak a$.
\end{remark}

The Bott-Samelson is a resolution for the Schubert cells in the affine Grassmannian. It is a smooth projective variety of dimension ${\rm dim}(\hat{\Sigma}(\gamma_\lambda))=2 \left\langle \lambda, \rho \right\rangle$. We will write a point in this variety as $g=[g_0,\ldots,g_p]$ and will call it a gallery. By the above Remark \ref{BS_gall} this naming is reasonable. The natural multiplication map then provides us with a $T$-equivariant map
\begin{eqnarray*}
\pi:\hat{\Sigma}(\gamma_\lambda) & \rightarrow & X_\lambda = \overline{\grass_\lambda}\\
g & \mapsto & g_0g_1 \cdots g_p
\end{eqnarray*}
that is birational and proper.

The set of combinatorial galleries of type $t_{\gamma_\lambda}$ starting at the origin are the $T$-fixed points under the natural torus action. In addition we also have the action of $G(\ko)$ on the variety, this comes from the action on the building, as $G(\ko)$ acts by simplical maps. This action preserves the type of a gallery, and if a gallery is minimal, then so are all the galleries in the $G(\ko)$-orbit.

If we view the Bott-Samelson variety as the set of galleries, the map $\pi:\hat{\Sigma}(\gamma_\lambda) \rightarrow X_\lambda$ is the restriction of the projection on the last factor. This means that the minimal gallery $\gamma_\lambda$ is mapped to $L_\lambda$, hence the map $\pi$ induces a morphism between the $G(\ko)$-orbit of $\gamma_\lambda$ and $\grass_\lambda$. As all minimal galleries lie in the same $G(\ko)$-orbit, see \cite[§ 7]{GauLit}, this induces a bijection between these and the open orbit $\grass_\lambda$ in $X_\lambda$ via $\pi$.

As the retractions were defined for alcoves, they are naturally extended to galleries by applying them to every alcove of the gallery. Since the retractions preserve the type of a face and hence of a gallery, the image of an arbitrary gallery of type $t_{\gamma_\lambda}$ is a combinatorial gallery of the same type. By using this fact one has the following property.

\begin{proposition}\textbf{(\cite[§7, Prop. 5]{GauLit})}\label{BB-cells}
The retraction $r_w$ with centre $w$ at infinity induces a map $\hat{r}_w:\hat{\Sigma}(\gamma_\lambda) \rightarrow \Gamma(\gamma_\lambda)$. The fibres $C_w(\delta)=\hat{r}_w^{-1}(\delta)$, $\delta \in \Gamma(\gamma_\lambda)$, are endowed with the structure of a locally closed subvariety, which are isomorphic to some affine spaces.

The $C_w(\delta)$ are the Bialynicki-Birula cells $\{x \in \hat{\Sigma}(\gamma_\lambda) \mid \limes_{s \rightarrow 0} s^\eta.x = \delta \}$ of centre $\delta$ in $\hat{\Sigma}(\gamma_\lambda)$, associated to a generic one-parameter subgroup $\eta$ of $T$ in the Weyl chamber $w\mathfrak{C}_f$.
\end{proposition}

This is a slight modification of the proposition in \cite{GauLit} as it uses retractions for every element of the Weyl group and not for $w_0$ alone. This does not change anything for the proof of the proposition as any other Weyl group element besides $w_0$ just corresponds to a different choice of a Borel and a different set of positive roots.

As mentioned above we can identify the open orbit $\grass_\lambda$ with the set of minimal galleries in $\hat{\Sigma}(\gamma_\lambda)$. One of the main results in \cite{GauLit} is the following theorem, which relates the intersection with the semi-infinite orbit with the retractions at infinity.

\begin{theorem} \textbf{(\cite[§7, Theorem 3]{GauLit})}
The restriction of $\hat{r}_w$ induces a map $r_w:\grass_\lambda \rightarrow \Gamma^w(\gamma_\lambda)$. Furthermore the union $\bigcup_\delta r_w^{-1}(\delta)$ of the fibres over all galleries in $\Gamma^w(\gamma_\lambda)$ with target $\nu$ is the intersection $S_\nu^w \cap \grass_\lambda$.
\end{theorem}

To give a better description of the fibres, we introduce an open covering of the Bott-Samelson variety.

For a combinatorial gallery $\delta =[\delta_0,\ldots,\delta_p] \in \Gamma(\gamma_\lambda)$ and $w \in W$ we define 
$$\ku_0^w = \prod_{\begin{array}{c}\beta \in w(\Phi^+), \\ \delta_0^{-1}\beta < 0 \end{array}}U_\beta \cdot \prod_{\begin{array}{c}\alpha \in w(\Phi^-), \\ \delta_0^{-1}\alpha < 0 \end{array}} U_\alpha \cdot \delta_0. $$
Furthermore for $1 \leq j \leq p$ we define
$$\ku_j = \left\{ \begin{array}{cl} U_{\alpha_{i_j}}s_{i_j} & \text{if } \delta_j = s_{i_j} \\
U_{-\alpha_{i_j}} & otherwise. \end{array} \right.$$

The following definition is a modification of the one in \cite[§ 8]{GauLit}, it is just the translation of the definition given there to the case of the simplified definition of the Bott-Samelson we introduced earlier.

\begin{definition} \textbf{(\cite[§8, Def. 23]{GauLit})}
Let $\delta =[\delta_0,\ldots ,\delta_p] \in \Gamma(\gamma_\lambda)$ be a combinatorial gallery and $w \in W$. Then we define the subvariety $\ku_\delta^w$ of $\hat{\Sigma}(\gamma_\lambda)= G(\ko) \times_{\ki} \kp_1 \times_{\ki} \ldots \times_{\ki} \kp_p / \ki$ as follows
$$ \ku_\delta^w= \ku_0^w \times \ku_1 \times \ldots \times \ku_p.$$
\end{definition}

The collection of these subvarieties for a given $w \in W$ form an affine open covering of $\hat{\Sigma}(\gamma_\lambda)$.

\begin{remark}
By Proposition \ref{BB-cells}, it is evident that for a combinatorial gallery $\delta$, the cell $C_w(\delta)$ lies inside $\ku_\delta^w$.
\end{remark}

\begin{remark}
If we talk about galleries in such an open piece $\ku_\delta^w$ of the Bott-Samelson variety corresponding to a combinatorial gallery, we always mean that the gallery is written in the corresponding chart as given above, i.e., we have the form
$$g_0 = \prod_{\begin{array}{c}\beta \in w(\Phi^+), \\ \delta_0^{-1}\beta < 0 \end{array}}x_{\beta}(a_\beta) \cdot \prod_{\begin{array}{c}\alpha \in w(\Phi^-), \\ \delta_0^{-1}\alpha < 0 \end{array}} x_\alpha(a_\alpha) \cdot \delta_0, $$
for some arbitrary complex parameters and for $1 \leq j \leq p$ we have
$$g_j = \left\{ \begin{array}{cl} x_{\alpha_{i_j}}(a_j)s_{i_j} & \text{if } \delta_j = s_{i_j} \\
x_{-\alpha_{i_j}}(a_j) & otherwise, \end{array} \right.$$
for some $a_j \in \IC$.
\end{remark}

The structure of the fibres of the retraction at infinity can then be described more precise in the following way, stated in \cite{GauLit}.

\begin{theorem} \label{BB_cell} \textbf{(\cite[§11, Theorem 4]{GauLit})}
Let $\lambda \in X^\vee_+$ and let $\delta=[\delta_0,\delta_1,\ldots,\delta_p] \in \Gamma^+(\gamma_\lambda)$. Then $\hat{r}_{w_0}^{-1}(\delta)$ is a subvariety of $\ku_\delta$ isomorphic to a product of $\IC$'s and $\IC^*$'s. More precisely, the fibre consists of all galleries $[g_0,g_1,\ldots,g_p]$ such that
$$g_0 \in \prod_{\beta < 0, \delta_0^{-1}(\beta)<0} U_\beta \cdot \delta_0 \text{ and } g_j = \left\{ \begin{array}{cl} \delta_j & \text{if } j \not\in J_{-\infty}(\delta) \\ x_{-\alpha_{i_j}}(a_j), a_j \neq 0& \text{if } j \in J_{-\infty}^-(\delta) \\ x_{\alpha_{i_j}}(a_j)s_{i_j} & \text{if } j \in  J_{-\infty}^+(\delta) \end{array} \right. $$
\end{theorem}

By definition $\overline{S_\nu^w \cap \grass_\lambda}$ is the union of $Z(\delta)=\overline{r_w^{-1}(\delta)}$ for $\delta \in \Gamma^w(\gamma_\lambda,\nu)$. Since the intersection $S_\nu^w \cap \grass_\lambda$ is equidimensional, by \cite{MirVil2}, we can reduce this to the union of those galleries whose fibre has the maximal dimension. Hence, using a statement from \cite{GauLit} that a gallery and its fibre have the same dimension, we reduce to
$$ \overline{S_\nu^w \cap \grass_\lambda} = \bigcup _{\delta \in \Gamma^w_{LS}(\gamma_\lambda,\nu)}Z(\delta). $$
As seen in the theorem above an open subset of the MV-cycle can be given in a quite explicit way.

This can be used to describe the MV-cycle and MV-polytope by using the retractions in the directions for all Weyl group elements and using the GGMS-strata. As a GGMS-stratum that lies inside $\grass_\lambda$ and is the intersection of $S_\nu$ with other semi-infinite cells defines an open subset of a MV-cycle of $\overline{S_\nu \cap \grass_\lambda}$, we obtain the following proposition by combining the above mentioned results.

\begin{proposition} \label{retraction_galleries}
Let $A(\mu_\bullet) \subset \overline{S_\nu \cap \grass_\lambda}$ be an MV-cycle of coweight $(\lambda,\nu)$. 

Then there exist an open subset $O \subset A(\mu_\bullet) \cap \grass_\lambda$ and galleries \break $\delta_w \in \Gamma^w_{LS}(\gamma_\lambda,\mu_w)$ for each element of the Weyl group, such that $r_w(x)=\delta_w$ for all $x \in O$.
\end{proposition}
\begin{proof}
For each $w \in W$ there exists $\delta_w \in \Gamma_{LS}^w(\gamma_\lambda,\mu_w)$ such that $r_w^{-1}(\delta_w) \subset A(\mu_\bullet) \cap \grass^\lambda$ and it is an open subset. Hence 
$$ \bigcap_{w \in W} r_w^{-1}(\delta_w) $$
is dense in $A(\mu_\bullet) \cap \grass^\lambda$.
\end{proof}

For a fixed LS-gallery $\delta$, we thus have to construct galleries $\delta_w$ for $w \in W$ with the properties from the proposition. By using the coweights of these galleries $d_\bullet=(wt(\delta_w))_{w \in W}$ we then obtain the MV-cycle as the closure of the GGMS-stratum $A(d_\bullet)$. In the next part we will construct these galleries.


\section{Sections}\label{section_5}
We now want to divide an LS-gallery into parts that will be called sections and study the behaviour of these sections with respect to the root operators that were defined in the last part. Everything stated about sections would, in a very similar matter, also work for the path model described in \cite{Lit4}, \cite{Lit5}, and \cite{Lit6}.

\subsection*{Stable and directed sections}

We start by analysing certain alcoves of an LS-gallery that, when applying the operators $e_\alpha$ or $f_\alpha$, for a chosen simple root $\alpha$, successively to the gallery, are not reflected at any hyperplane, but only translated.

In the following let $\delta = ( \Delta_0' \subset \Delta_0 \supset \Delta_1' \subset \ldots \supset \Delta_p' \subset \Delta_p \supset \Delta_{p+1}' )$ be a gallery and unless otherwise stated $\delta \in \Gamma_{LS}^+(\gamma_\lambda)$.

\begin{definition}\label{stable}
Let $\alpha$ be a simple root. An alcove $\Delta_i$ of $\delta$ is called \textit{$\alpha$-stable} if there exists $m \in \IZ$ and $0 \leq l \leq i < r \leq p+1$, such that
\begin{enumerate}
\item $\Delta_{r}' \subset H_{\alpha,m}$,
\item $\Delta_{l}' \subset H_{\alpha,m}$, and
\item for all $l \leq s < r$ : $\Delta_s' \not\subset H_{\alpha,m-1}$.
\end{enumerate}

If the index $m$ is minimal such that the indices $l$ and $r$ exist we say that $\Delta_i$ is \textit{$\alpha$-stable at $m$}, in this case we denote by $r_\alpha(\Delta_i)$ the minimal index such that properties (i) and (iii) are fulfilled and by $l_\alpha(\Delta_i)$ the maximal index such that properties (ii) and (iii) are fulfilled.

Let $R_{\alpha,m}(\delta)$ be the set of all indices $j$, such that $\Delta_j$ is $\alpha$-stable at $m$ and 
$$R_{\alpha}(\delta):=\bigcup_{m \in \IZ} R_{\alpha,m}(\delta),$$
the set of all indices corresponding to $\alpha$-stable alcoves.
\end{definition}

We first look at some basic properties of these indices.

\begin{lemma} \label{sec_stable}
If $i \in R_{\alpha,m}(\delta)$, then $j \in R_{\alpha,m}(\delta)$ for all $l_\alpha(\Delta_i) \leq j < r_\alpha(\Delta_i)$.
\begin{proof}
We first assume that $j \notin R_{\alpha,m}(\delta)$. As $l_\alpha(\Delta_i) \leq j < r_\alpha(\Delta_i)$ and they satisfy condition (i), (ii), and (iii) of Definition \ref{stable}, we know that $j \in R_{\alpha}(\delta)$. Hence by assumption $j \in R_{\alpha,n}(\delta)$ for $n < m$. Since $\Delta_s' \not\subset H_{\alpha,m-1}$ for $l_\alpha(\Delta_i) \leq s < r_\alpha(\Delta_i)$ this means that $l_\alpha(\Delta_j) < l_\alpha(\Delta_i)$ and $r_\alpha(\Delta_i) < r_\alpha(\Delta_j)$, which implies that $i \in R_{\alpha,n}(\delta)$, for $n<m$, which is a contradiction.
\end{proof}
\end{lemma}

In addition we also want to see how the neighbourhood of indices not $\alpha$-stable at any $m$ looks like.

\begin{lemma} \label{sec_directed}
If $i \notin R_{\alpha}(\delta)$ and $j \leq i$ is maximal such that $\Delta_j' \subset H_{\alpha,m}$ for some $m \in \IZ$ and $k > i$ is minimal such that $\Delta_k' \subset H_{\alpha,m'}$ for some $m' \in \IZ$, then $\{j,\ldots, k-1 \} \cap R_{\alpha}(\delta) = \emptyset$ and furthermore $m-m' = \pm 1$.
\end{lemma}
\begin{proof}
If there exists $s \in {j,\ldots, k-1 }$ and an $n \in \IZ$ such that $s \in R_{\alpha,n}(\delta)$, then $l_\alpha(\Delta_s) \leq j$ and $r_\alpha(\Delta_s) \geq k$ and thus $l_\alpha(\Delta_s) \leq i < r_\alpha(\Delta_s)$. By Lemma \ref{sec_stable} this implies $i \in R_{\alpha,n}(\delta)$ which is a contradiction. \\
It is obvious that $m-m' \in \{-1, 0, 1\}$. Suppose $m-m'=0$, then $i \in R_{\alpha,m}(\delta)$ unless $m$ is not minimal, as the pair $(j,k)$ satisfies condition (i) - (iii) of Definition \ref{stable}. But by Remark \ref{non_minimum} this shows that $i$ is $\alpha$-stable, which is a contradiction.
\end{proof}

\begin{definition}\label{section}
For $\delta \in \Gamma_{LS}^+(\gamma_\lambda)$, an interval $[i,j] \subset [0,p]$ is called
\begin{itemize}
\item \textit{$\alpha$-directed section} if there exists an $m \in \IZ$, such that for all $k \in [i,j-1]$
$$ 
\left\{\begin{array}{r}
\Delta_k' \subset H_{\alpha,m} \Rightarrow k=i \\
\Delta_k' \subset H_{\alpha,m+1} \Rightarrow k=j 
\end{array}\right\} \wedge k \notin R_\alpha(\delta)
$$
holds.
\item \textit{$(-\alpha)$-directed section} if there exists an $m \in \IZ$, such that for all $k \in [i,j-1]$
$$\left\{\begin{array}{r}
\Delta_k' \subset H_{\alpha,m} \Rightarrow k=i \\
\Delta_k' \subset H_{\alpha,m-1} \Rightarrow k=j 
\end{array}\right\} \wedge k \notin R_\alpha(\delta)
$$
holds.
\item \textit{$\alpha$-stable section} at $m$ if there exists a $k \in [i,j-1]$, such that $k \in R_{\alpha,m}(\delta)$ and $r_\alpha(\Delta_k)=j$ and $l_\alpha(\Delta_k)=i$.
\end{itemize}
\end{definition}

\begin{remark}
For the directed section the inclusion conditions also imply that no $\Delta_k'$ lies in any $\alpha$-hyperplane as long as $k$ is not equal to $i$ or $j$.
\end{remark}

This definition allows us to divide our gallery into parts, each being one of the three sections above.

\begin{lemma} \label{partition}
For $\delta \in \Gamma_{LS}^+(\gamma_\lambda)$, there exists a unique partition $i_1 < i_2 < \ldots < i_t$, such that each interval $[i_k,i_{k+1}]$ is an $\alpha$-directed, $(-\alpha)$-directed, or $\alpha$-stable section.
\end{lemma}
\begin{proof}
The existence of a partition follows from Lemma \ref{sec_stable} and Lemma \ref{sec_directed} and the uniqueness is obvious for the directed sections and for the stable sections it is implicated by the following remark.
\end{proof}

\begin{remark}
For an $\alpha$-stable section, it is obvious that, that neither $[i_{k-1},i_{k}]$ nor $[i_{k+1},i_{k+2}]$ are $\alpha$-stable, since $\alpha$-stable sections can not be enlarged by definition.
\end{remark}

\begin{remark}\label{order_directed}
Let $i < j < k < l$, such that $[i,j]$ is an $\alpha$-directed section. Assume that that $[j,k]$ is a $(-\alpha)$-directed section, but this implies that all $\Delta_s$ for $s \in [i,k]$ is $\alpha$-stable, which is a contradiction. Alternatively assume that $[j,k]$ is a stable section and $[k,l]$ is a $(-\alpha)$-directed section. As before this implies that $\Delta_s$ for $s \in [i,l]$ is $\alpha$-stable, which is again a contradiction.

For our partition this means that all $\alpha$-directed section appear after the $(-\alpha)$-directed ones, with $\alpha$-stable sections appearing anywhere in between.
\end{remark}

\subsection*{Sections and root operators}

We want to see how the defined partition of our gallery interacts with the operators $f_\alpha$ and $e_\alpha$.

\begin{lemma} \label{f_not_stable}
If $\left\langle wt(\delta),\alpha \right\rangle-{\boldsymbol{\mathfrak m}_\alpha^\delta} \geq 1$. Then $\Delta_s \notin R_\alpha(\delta)$ for all ${\boldsymbol{\mathfrak j}_\alpha^\delta} \leq s < {\boldsymbol{\mathfrak k}_\alpha^\delta}$.
\end{lemma}
\begin{proof}
Let us first assume that $\Delta_s$ is $\alpha$-stable at $n > {\boldsymbol{\mathfrak m}_\alpha^\delta}$. Then there exists $l \leq s$, such that $\Delta_l' \subset H_{\alpha,n}$ and $\Delta_t' \not\subset H_{\alpha,n-1}$ for $l \leq t \leq s$, in particular it implies that $\Delta_t' \not\subset H_{\alpha,{\boldsymbol{\mathfrak m}_\alpha^\delta}}$ for $l \leq t \leq s$. Hence ${\boldsymbol{\mathfrak k}_\alpha^\delta} \leq l$ as both $l$ and ${\boldsymbol{\mathfrak k}_\alpha^\delta}$ are strictly bigger than $j$, but ${\boldsymbol{\mathfrak k}_\alpha^\delta}$ is minimal with the property that $\Delta_{\boldsymbol{\mathfrak k}_\alpha^\delta}' \subset H_{\alpha,m+1}$. But this would be a contradiction to $l \leq s$ as ${\boldsymbol{\mathfrak j}_\alpha^\delta} \leq s < {\boldsymbol{\mathfrak k}_\alpha^\delta} \leq l$.

The cases that $\Delta_s$ is $\alpha$ stable at $n < {\boldsymbol{\mathfrak m}_\alpha^\delta}$ or at ${\boldsymbol{\mathfrak m}_\alpha^\delta}$ are dealt with in a very similar and straightforward fashion.
\end{proof}

In other words, we have shown a bit more than stated in the lemma, we have shown that $[{\boldsymbol{\mathfrak j}_\alpha^\delta},{\boldsymbol{\mathfrak k}_\alpha^\delta}]$ is an $\alpha$-directed section of $\delta$. 

We now want to prove, that if we can apply the operator $f_\alpha$ to a gallery, an $\alpha$-stable alcove remains $\alpha$-stable and an alcove that is $\alpha$-stable afterwards was obtained from one.

\begin{proposition} \label{f_alpha_stable}
If $\left\langle wt(\delta),\alpha \right\rangle-{\boldsymbol{\mathfrak m}_\alpha^\delta} \geq 1$. Then for all $i$, it holds
$$i \in R_{\alpha}(\delta) \Longleftrightarrow i \in R_{\alpha}(f_\alpha\delta).$$
\end{proposition}
\begin{proof}
Since $f_\alpha\delta = (F_f \subset \widetilde{\Delta_0} \supset \widetilde{\Delta_1'} \subset \ldots \supset \widetilde{\Delta_p'} \subset \widetilde{\Delta_p} \supset \widetilde{\Delta_{p+1}'})$ is defined by the assumptions on ${\boldsymbol{\mathfrak m}_\alpha^\delta}$. We start with $i < {\boldsymbol{\mathfrak j}_\alpha^\delta}$.
\begin{enumerate}
\item Let $i \in R_{\alpha,n}(\delta)$. It is clear that $n \geq {\boldsymbol{\mathfrak m}_\alpha^\delta}$ and as $l_\alpha(\Delta_i) \leq i < {\boldsymbol{\mathfrak j}_\alpha^\delta}$ there are two possibilities. Either $n={\boldsymbol{\mathfrak m}_\alpha^\delta}$, then $r_\alpha(\Delta_i) \leq {\boldsymbol{\mathfrak j}_\alpha^\delta}$ as ${\boldsymbol{\mathfrak j}_\alpha^\delta}$ is maximal with the property that $\Delta_{\boldsymbol{\mathfrak j}_\alpha^\delta}' \subset H_{\alpha,{\boldsymbol{\mathfrak m}_\alpha^\delta}}$. Or on the other hand $n > {\boldsymbol{\mathfrak m}_\alpha^\delta}$, then $r_\alpha(\Delta_i) < {\boldsymbol{\mathfrak j}_\alpha^\delta}$ because of the property that $\Delta_t' \not\subset H_{\alpha,{\boldsymbol{\mathfrak m}_\alpha^\delta}-1}$ for all $l_\alpha(\Delta_i) \leq t < r_\alpha(\Delta_i)$.

Thus $\Delta_{l_\alpha(\Delta_i)}, \ldots, \Delta_{r_\alpha(\Delta_i)-1}$ are not changed by $f_\alpha$. As $\widetilde{\Delta_{\boldsymbol{\mathfrak k}_\alpha^\delta}'}$ is the only face with the property $ \widetilde{\Delta_{\boldsymbol{\mathfrak k}_\alpha^\delta}'} \subset H_{\alpha,{\boldsymbol{\mathfrak m}_\alpha^\delta}-1}$ and as no face between $\widetilde{\Delta_{\boldsymbol{\mathfrak j}_\alpha^\delta}'}$ and $\widetilde{\Delta_{\boldsymbol{\mathfrak k}_\alpha^\delta}'}$ lies in any $\alpha$ wall, we have that $\{ l_\alpha(\Delta_i), \ldots, r_\alpha(\Delta_i)-1 \} \subset R_{\alpha,n}(f_\alpha\delta)$ and especially $i \in R_{\alpha,n}(f_\alpha\delta)$.
\item If on the other hand $i \in R_{\alpha,n}(f_\alpha\delta)$ then $n \geq {\boldsymbol{\mathfrak m}_\alpha^\delta}$ and like in the previous case there exists $r_\alpha(\widetilde{\Delta_i})$ and $l_\alpha(\widetilde{\Delta_i})$ both lower or equal to ${\boldsymbol{\mathfrak j}_\alpha^\delta}$ as $\widetilde{\Delta_{\boldsymbol{\mathfrak k}_\alpha^\delta}'}$ is the only face that lies in $H_{\alpha,{\boldsymbol{\mathfrak m}_\alpha^\delta}-1}$. But this means that $\widetilde{\Delta_s}=\Delta_s$ for all $l_\alpha(\widetilde{\Delta_i}) \leq s < r_\alpha(\widetilde{\Delta_i})$ and thus $i \in R_{\alpha,n}(\delta)$.
\end{enumerate}

Next we look at ${\boldsymbol{\mathfrak j}_\alpha^\delta} \leq i < {\boldsymbol{\mathfrak k}_\alpha^\delta}$. By Lemma \ref{f_not_stable} it is already obvious that $i \notin R_\alpha(\delta)$, thus it remains to show $i \notin R_\alpha(f_\alpha\delta)$. This is essentially the same argument, we suppose that $i \in R_{\alpha,n}(f_\alpha\delta)$ for some $n \in \IZ$. Since the first face after $\widetilde{\Delta_i'}$ lying in an $\alpha$-wall is $\widetilde{\Delta_{\boldsymbol{\mathfrak k}_\alpha^\delta}'}$ it is clear that by condition (iii) of Definition \ref{stable} we have $n={\boldsymbol{\mathfrak m}_\alpha^\delta}-1$. But by definition of $f_\alpha$, $\widetilde{\Delta_{\boldsymbol{\mathfrak k}_\alpha^\delta}'}$ is the only face that lies in the $\alpha$-wall with height ${\boldsymbol{\mathfrak m}_\alpha^\delta}-1$ which is a contradiction to condition (ii) of Definition \ref{stable} as the index $l_\alpha(\widetilde{\Delta_i})$ cannot exist.

The case $i \geq {\boldsymbol{\mathfrak k}_\alpha^\delta}$ is proven in a very similar manner to the first part.
\end{proof}

\begin{remark} \label{e_alpha_stable}
The analog for $e_\alpha$ holds as well and is proven in the same way.
\end{remark}

It is easy to see that the condition for $f_\alpha$ to be defined, is related to the number of $\alpha$-directed sections.

\begin{proposition} \label{f_alpha_defined_directed}
It holds:
$$ f_\alpha\delta \text{ is defined } \Longleftrightarrow \text{ there exists an $\alpha$-directed section}. $$
\end{proposition}
\begin{proof}
If $f_\alpha\delta$ is defined, then by Lemma \ref{f_not_stable}, the interval $[{\boldsymbol{\mathfrak j}_\alpha^\delta},{\boldsymbol{\mathfrak k}_\alpha^\delta}]$ is an $\alpha$-directed section.

On the other hand, let $[s,t]$ be the first $\alpha$-directed section and $\Delta_s' \subset H_{\alpha,n}$, then by Remark \ref{order_directed}, it is easy to see that ${\boldsymbol{\mathfrak m}_\alpha^\delta}=n$, ${\boldsymbol{\mathfrak j}_\alpha^\delta}=s$, and ${\boldsymbol{\mathfrak k}_\alpha^\delta}=t$. Since $[s,t]$ is $\alpha$-directed, we also know that $\left\langle wt(\delta),\alpha \right\rangle > {\boldsymbol{\mathfrak m}_\alpha^\delta}$. Hence $f_\alpha$ is defined.
\end{proof}

\begin{remark}
As one can also see in the proof of Proposition \ref{f_alpha_defined_directed}, if $f_\alpha$ is defined it will reflect the first $\alpha$-directed section at the wall orthogonal to $\alpha$ of minimal height and thus produce a $(-\alpha)$-directed section at that part of the gallery and translate or leave invariant the rest.
\end{remark}

\begin{proposition} \label{e_alpha_defined_directed}
It holds:
$$ e_\alpha\delta \text{ is defined } \Longleftrightarrow \text{ there exists an $(-\alpha)$-directed section}. $$
\end{proposition}
\begin{proof}
The proof is ananlogous to the one of \ref{f_alpha_defined_directed}.
\end{proof}

Combining these two propositions we obtain the following relation between the partitions and the operators $e_\alpha$ and $f_\alpha$.

\begin{theorem} \label{op_theorem}
Let $\delta \in \Gamma_{LS}^+(\gamma_\lambda)$ and $i_1 < i_2 < \ldots i_k$ the corresponding partition into $\alpha$-directed, $(-\alpha)$-directed, and $\alpha$-stable sections. Then the following holds:

\textbf{If $e_\alpha\delta$ is defined:} Let $[i_s,i_{s+1}]$ be the last $(-\alpha)$-directed section of $\delta$. Then $e_\alpha\delta$ also has the partition $i_1 < i_2 < \ldots < i_k$ and all sections are of the same type as before, except that $[i_s,i_{s+1}]$ is an $\alpha$-directed section for $e_\alpha \delta$.
Furthermore, if $\Delta_{i_{s+1}+1}' \subset H_{\alpha,{\boldsymbol{\mathfrak m}_\alpha^\delta}}$ and $e_\alpha\delta = (\widetilde{\Delta_0'} \subset \widetilde{\Delta_0} \supset \widetilde{\Delta_1'} \subset \ldots \supset \widetilde{\Delta_p'} \subset \widetilde{\Delta_p} \supset \widetilde{\Delta_{p+1}'} )$, then 
\[ \widetilde{\Delta_j} = \left\{\begin{array}{ll}
\Delta_j & \text{ if } j < i_s, \\
s_{\alpha,{\boldsymbol{\mathfrak m}_\alpha^\delta}}\Delta_j & \text{ if } i_s \leq j \leq i_{s+1}, \\
t_{\alpha^\vee}\Delta_j & j > i_{s+1}.
\end{array} \right. \]
Furthermore $\varepsilon_\alpha(\delta)= \#\{ (-\alpha)-\text{directed sections} \}$.

\textbf{If $f_\alpha\delta$ is defined:} Let $[i_s,i_{s+1}]$ be the first $\alpha$-directed section of $\delta$. Then $f_\alpha\delta$ also has the partition $i_1 < i_2 < \ldots < i_k$ and all sections are of the same type as before, except that $[i_s,i_{s+1}]$ is now a $(-\alpha)$-directed section for $f_\alpha \delta$.
Furthermore, if $\Delta_{i_s+1}' \subset H_{\alpha,{\boldsymbol{\mathfrak m}_\alpha^\delta}}$ and $f_\alpha\delta = (\widetilde{\Delta_0'} \subset \widetilde{\Delta_0} \supset \widetilde{\Delta_1'} \subset \ldots \supset \widetilde{\Delta_p'} \subset \widetilde{\Delta_p} \supset \widetilde{\Delta_{p+1}'} )$, then 
\[ \widetilde{\Delta_j} = \left\{\begin{array}{ll}
\Delta_j & \text{ if } j < i_s, \\
s_{\alpha,{\boldsymbol{\mathfrak m}_\alpha^\delta}}\Delta_j & \text{ if } i_s \leq j \leq i_{s+1}, \\
t_{-\alpha^\vee}\Delta_j & j > i_{s+1}.
\end{array} \right. \]
Furthermore $\varphi_\alpha(\delta)= \#\{ \alpha-\text{directed sections} \}$. 
\end{theorem}

\begin{definition}
The \textit{flipping} $(\delta)_{-\alpha}$ with respect to $\alpha$ of $\delta$ is defined as the gallery $(\widetilde{\Delta_0'} \subset \widetilde{\Delta_0} \supset \widetilde{\Delta_1'} \subset \ldots \supset \widetilde{\Delta_p'} \subset \widetilde{\Delta_p} \supset \widetilde{\Delta_{p+1}'} )$:
$$\widetilde{\Delta_r} = \left\{\begin{array}{ll}
s_{\alpha,m}\Delta_r & \text{ if } \Delta_r \in R_{\alpha,m}(\delta), \\
\Delta_r & \text{ otherwise}.
\end{array} \right. $$
This produces a well defined gallery by the definition of stableness in \ref{stable} and Lemma \ref{sec_stable}.
\end{definition}

It should be noted that, in general, when applying the flipping operator $(\cdot)_{-\alpha}$ to a gallery in $\Gamma_{LS}^+(\gamma_\lambda)$, the new gallery is not an LS-gallery, neither for $\mathfrak C^{\infty}_{w_0}$ nor for $\mathfrak C^{\infty}_{w_0s_\alpha}$. We need the following special situation.

\begin{theorem} \label{f_e_corres}
$(f_\alpha^{\max}(\delta))_{-\alpha}=s_\alpha.e_\alpha^{\max}(\delta)$, where $s_\alpha$ operates on the \break gallery by using the identification $\ka = X^\vee \otimes \IR$.
\end{theorem}
\begin{proof}
This follows immediately from Theorem \ref{op_theorem}.
\end{proof}

\begin{remark}
This implies that $(f_\alpha^{\max}(\delta))_{-\alpha} \in \Gamma_{LS}^{s_\alpha}(\gamma_\lambda)$.
\end{remark}

\begin{definition} \label{xiidef}
Let $w \in W$ and $w^{\underline{i}} = s_{i_1} \ldots s_{i_n}$ a reduced decomposition of $w$ and $w_k^{\underline{i}}=s_{i_1} \ldots s_{i_k}$ for $0 \leq k \leq n$, we define the following series of galleries for $1 \leq k \leq n$
$$\delta_0^{\underline{i}} := \delta \text{ and } \delta_k^{\underline{i}} := w_k^{\underline{i}}\left(e_{\alpha_{i_k}}^{\max}((w_{k-1}^{\underline{i}})^{-1}\delta_{k-1}^{\underline{i}})\right) = w_{k-1}^{\underline{i}}\left(f_{\alpha_{i_k}}^{\max}((w_{k-1}^{\underline{i}})^{-1}\delta_{k-1}^{\underline{i}})\right)_{-\alpha_{i_k}}.$$
We define $\Xi_w^{\underline{i}}(\delta):=\delta_n^{\underline{i}}$.
\end{definition}

This definition is independent of the choice of the reduced decomposition, as shown by the following proposition.

\begin{proposition} \label{xi_independent}
Let $w$, $w_k^{\underline{i}}$, and $\delta_k^{\underline{i}}$ be as in Definition \ref{xiidef}. We also set $\widetilde{\delta}_0^{\underline{i}}:=\delta$ and $\widetilde{\delta}_k^{\underline{i}}:=e_{\alpha_{i_k}}^{\max}(\widetilde{\delta}_{k-1}^{\underline{i}})$ for $1 \leq k \leq n$. Then it holds:
\begin{enumerate}
\item $\delta_l^{\underline{i}}=w_l^{\underline{i}}.\widetilde{\delta}_l^{\underline{i}}$ for $0 \leq l \leq n$,
\item $\widetilde{\delta}_n^{\underline{i}}$ is independent of the choice of the reduced decomposition.
\end{enumerate}
\end{proposition}
\begin{proof}
Part (i) is shown by a simple induction on $l$, being true for $l=0$ by definition.

Part (ii) follows immediately from \cite{Lit}.
\end{proof}

Thus the operators $\Xi_w^{\underline{i}}$ are independent of the reduced decomposition.

\begin{definition} \label{xi_def}
For $w \in W$, we define the \textit{vertex gallery} $\Xi_w(\delta)$ of $\delta$ with respect to $w$ to be $\Xi_w^{\underline{i}}(\delta)$ for an arbitrary reduced decomposition $\underline{i}$ of $w$.
\end{definition}

An easy consequence of Proposition \ref{xi_independent} is the following.

\begin{corollary}
The gallery $\Xi_w(\delta) \in \Gamma_{LS}^w(\gamma_\lambda)$ for all $w \in W$.
\end{corollary}

\begin{remark} \label{recursive_xi}
To write this in a more convenient way, we define for $\delta \in \Gamma_{LS}^+(\gamma_\lambda)$ and a reduced expression $w=s_{i_1} \ldots s_{i_r}$, $e_w^{\rm max}(\delta):=e_{\alpha_{i_r}}^{\rm max} \cdots e_{\alpha_{i_1}}^{\rm max}(\delta)$, which is independent by Proposition \ref{xi_independent}. Thus we have
$$ \Xi_w(\delta) = w e_w^{\rm max}(\delta). $$
This implies immediately that for $w \in W$ and $\alpha \in \Phi^+$, such that $l(ws_\alpha)>l(w)$, it holds:
$$ \Xi_{ws_\alpha}(\delta) = \Xi_{s_{w\alpha}}(\Xi_w(\delta)). $$
\end{remark}

This is the reason why we can hope to use an inductive argument in Section \ref{section_6}.

\begin{remark} \label{xi_delta_difference}
For $\delta = [\delta_0, \ldots, \delta_p]$ and $\Xi_{s_\alpha}(\delta) = [\delta_0',\ldots,\delta_p']$ it is obvious by definition that
$$ \delta_i' = \left\{\begin{array}{ll}
s_{\alpha_i} \delta_i & \text{if } i={\boldsymbol{\mathfrak k}_\alpha^\delta} \text{ and } {\boldsymbol{\mathfrak k}_\alpha^\delta} < p\\
s_{\alpha_i} \delta_i & \text{if $i<{\boldsymbol{\mathfrak k}_\alpha^\delta}$ and there exists $t \leq {\boldsymbol{\mathfrak k}_\alpha^\delta}$ such that $[i,t]$ is an $\alpha$-stable section} \\
s_{\alpha_i} \delta_i & \text{if $i<{\boldsymbol{\mathfrak k}_\alpha^\delta}$ and there exists $t \leq {\boldsymbol{\mathfrak k}_\alpha^\delta}$ such that $[t,i]$ is an $\alpha$-stable section} \\
\delta_i & \text{otherwise}
\end{array} \right. .$$
\end{remark}


\section{Retractions and Results}\label{section_6}
We introduced the notion of flipping to obtain a link between the positively folded LS-galleries and the retractions at infinity for all Weyl group elements. In this section we prove the main theorem about the retractions in the affine building when applied to a dense subset in a given cell $C(\delta)$, for a given combinatorial gallery $\delta \in \Gamma^+_{LS}(\gamma_\lambda)$. This relates the retractions with the combinatorial galleries $\Xi_w(\delta)$, defined in the last section, and will give a proof for the fact that these galleries can be used to define the MV-polytope, as well as the MV-cycle via the GGMS-strata.

\subsection*{Coordinate changes}

For $g \in C(\delta)$, let $g=g(a)=[g_0(\underline{a_\beta}),g_1(a_1)\ldots,g_p(a_p)]$ with 
$$g_0(\underline{a_\beta}) = \prod_{\beta < 0, \delta_0^{-1}(\beta)<0} x_\beta(a_\beta) \cdot \delta_0$$
and 
$$g_j(a_j) = \left\{ \begin{array}{ll} \delta_j & \text{if } j \not\in J_{-\infty}(\delta) \\ x_{-\alpha_{i_j}}(a_j), a_j \neq 0& \text{if } j \in J_{-\infty}^-(\delta) \\ x_{\alpha_{i_j}}(a_j)s_{i_j} & \text{if } j \in  J_{-\infty}^+(\delta) \end{array} \right. $$
Hence our gallery has coordinates $a_j$ for $1 \leq j \leq p$ as above and $a_\beta$ for negative roots $\beta$. In the latter case the coordinates for those roots $\beta < 0$ such that $\delta_0^{-1}(\beta)>0$ are just zero.

The first thing that needs to be done is to rewrite the gallery in the coordinates of $\ku_{\Xi_{s_\alpha}(\delta)}^{s_\alpha}$. This will not work for an arbitrary gallery $g$, we will need to impose some assumptions on the coordinates for this to be possible.

To rewrite the coordinates, we recall that by Remark \ref{xi_delta_difference} the galleries $\delta$ and $\Xi_{s_\alpha}(\delta)$ only differ at certain indices. Hence we have to proceed in two steps:
\begin{enumerate}
\item First we have to eliminate the folding at $\boldsymbol{\mathfrak t}_\alpha^\delta$ if this is less than $p$.
\item Second for every stable section occurring before the minimum we have to change the folding at the first index to a crossing and the crossing at the last index to a folding.
\end{enumerate}

The following proposition will be needed multiple times to obtain contradictions in the succeeding proofs.

\begin{proposition} \label{gallery_origin}
Let $\delta \in \Gamma(\gamma_\lambda)$, $1 \leq j \leq p$, and $\Delta_j' \subset H_{\beta,n}$ for some $\beta \in \Phi^+$ and some $n \in \IZ$. Then it holds
$$ \delta_0 \dots \delta_{j-1} (-\delta_j \alpha_{i_j}) = \left\{ \begin{array}{cc} -\beta+n\underline{\delta} & \text{ if } j \in J_{-\infty}(\delta) \\
\beta-n\underline{\delta} & \text{ if } j \not\in J_{-\infty}(\delta) \end{array} \right. $$
\end{proposition}
\begin{proof}
We only discuss the load bearing case, the other one is the same with opposite signs. We use the notations from the end of section \ref{section_2}. Let $H'$ be the wall, that contains the face of $\Delta_f$ of type $t_j'$. Then $H = \delta_0 \dots \delta_j H'$ contains $\Delta_j'$ by definition of the type. We can decompose the element $\delta_0 \dots \delta_j$ into $s^\mu w$ with $\mu \in \IZ\Phi^+$ and $w \in W$. As we are assuming that $j$ is load-bearing it holds that $\Delta_j=\delta_0 \dots \delta_j \Delta_f \subset H^+$, the closed positive half-space corresponding to $H$. This implies that $w \Delta_f \subset (t^{-\mu}H)^+$. We will now distinguish two cases.
\begin{enumerate}
\item \textbf{ Assume $\alpha_{i_j} \neq \alpha_0$:} Here $\Delta_j'$ and $H$ contain the unique vertex of $\Delta_j$ of type $S$. This implies that $t^{-\mu}H$ contains $0$ as $w0=0$. Thus $s^{-\mu}H$ is a hyperplane through the origin and hence $\Delta_f \subset (s^{-\mu}H)^+$. Since $\Delta_f$ and $w\Delta_f$ are included in $(s^{-\mu}H)^+$, we know $\beta:=w\alpha_{i_j} \in \Phi^+$ and we set $n = \left\langle \beta, \mu \right\rangle$, hence $H=H_{\beta,n}$. We conclude
\begin{eqnarray*}
(\delta_0 \dots \delta_{j-1})x_{\delta_j \alpha_{i_j}}(a)(\delta_0 \dots \delta_{j-1})^{-1} &=& s^\mu x_{w \alpha_{i_j}}(\pm a) s^{-\mu} \\
&=& x_\beta(\pm a s^{\left\langle \beta, \mu \right\rangle}) = x_\beta(\pm a s^n).
\end{eqnarray*}
For a positive folding at $j$, we obtain $\delta_0 \dots \delta_{j-1} (-\alpha_{i_j})= -\beta + n\underline{\delta}$, since $\delta_j \alpha_{i_j} = \alpha_{i_j}$. While for the case of a positive crossing we have $\delta_0 \dots \delta_{j-1} \alpha_{i_j}= -\beta + n\underline{\delta}$.

\item \textbf{ Assume $\alpha_{i_j} = \alpha_0$:} Now $H'=H_{\theta,1}$ and thus $t^{-\mu}H = wH' = H_{w\theta,1}$ and $w\Delta_f \subset (H_{w\theta,1})^+$. Thus we conclude that $w\theta \in \Phi^-$ and set $\beta = w(-\theta)$ and $n = \left\langle \beta, \mu \right\rangle -1$, hence $H=H_{-\beta,-n}=H_{\beta,n}$. Again we make the calculation
\begin{eqnarray*}
(\delta_0 \dots \delta_{j-1})x_{\delta_j \alpha_{i_j}}(a)(\delta_0 \dots \delta_{j-1})^{-1} &=& s^\mu x_{w \alpha_{i_j}}(\pm a t^{-1}) s^{-\mu}\\
&=& x_\beta(\pm a s^{\left\langle \beta, \mu \right\rangle -1}) = x_\beta(\pm a s^n).
\end{eqnarray*}
Hence we again obtain for a positive folding $\delta_0 \dots \delta_{j-1} (-\alpha_{i_j})= -\beta + n\underline{\delta}$ and for a positive crossing $\delta_0 \dots \delta_{j-1} \alpha_{i_j}= -\beta + n\underline{\delta}$.
\end{enumerate}
\end{proof}

We need to fix some additional notations to simplify the coming proofs.

\begin{notation} \label{Index_def}
\begin{enumerate}
\item We fix a simple root $\alpha$, an LS-gallery $\delta=[\delta_0,\delta_1,\ldots,\delta_p]$, and $g \in C(\delta)$ with the coordinates written as above. 
\item For an $\alpha$-stable section $[a,b]$ of $\delta$, stable at $n$, an index $i \in [a,b-1]$ is called \textit{critical} (or more precise $\alpha$-critical), if $\Delta_i' \subset H_{\alpha,n}$.
\item Let the intervals $[u_2,v_2], \ldots, [u_r,v_r]$ denote the $\alpha$-stable sections of $\delta$ such that $v_2 < {\boldsymbol{\mathfrak j}_\alpha^\delta}$ and $u_i > v_{i+1}$, hence the $\alpha$-stable sections that occur before ${\boldsymbol{\mathfrak j}_\alpha^\delta}$ in reverse order. For $[u_i,v_i]$ we denote by $C([u_i,v_i])$ the set of critical indices of this section.
\item We define $u_1 = \boldsymbol{\mathfrak t}_\alpha^\delta$ and $v_1 = p$ and denote by $C([u_1,v_1])$ those indices $u_1 \leq i \leq p$, such that $\Delta_i' \subset H_{\alpha,\boldsymbol{\mathfrak m}_\alpha^\delta}$.
\item We put $I^c = \bigcap_{1 \leq i \leq r} C([u_i,v_i])$ and for $1 \leq j \leq p$ we define $I^c_{<j} = \{i \in I^c \mid i < j \}$.
\end{enumerate}
\end{notation}

\begin{remark}
By definition $u_i \in C([u_i,v_i])$, for $1 \leq i \leq r$.
\end{remark}

We now have to deal with the different positions where we have to change the coordinates of our gallery. As we do not want to change the different positions in an arbitrary order, hence we fix a sequence of galleries $\kappa^i$, $1 \leq i \leq r$, lying "between" $\delta$ and $\Xi_{s_\alpha}(\delta)$. The galleries $\kappa^i=[\kappa^i_0,\kappa^i_1,\ldots,\kappa^i_p]$ are defined as follows:
$$ \kappa^0 = \delta,$$
$$ \kappa^1_j = \left\{ \begin{array}{cl} s_{\alpha_{i_j}}\kappa^{0}_j & j = u_1 \text{ and } u_1 < p\\ 
\delta_j & \text{otherwise} \end{array} \right., \text{ and}$$

$$ \kappa^i_j = \left\{ \begin{array}{cl} s_{\alpha_{i_j}}\kappa^{i-1}_j & j = u_i \text{ or } j = v_i \\ 
\kappa^{i-1}_j & \text{otherwise} \end{array} \right. \text{ for } 2 \leq i \leq r.$$

\begin{remark}
By construction $\kappa^r = \Xi_{s_\alpha}(\delta)$. One should also note that $\kappa^1$ and $\delta$ coincide if $\boldsymbol{\mathfrak t}_\alpha^\delta$.
\end{remark}

The following technical lemmas are key for the proofs of the succeeding propositions.

\begin{lemma} \label{new_terms_1}
Let $\delta \in \Gamma_{LS}^+(\gamma_\lambda)$, $\alpha$ simple root, and $0 < k < p$ such that
\begin{itemize}
\item $\delta$ has a positive folding at $k$,
\item $\Delta_k' \subset H_{\alpha,n}$ for some $n$, and
\item if for $k < j \leq p$ there exists $t$ with $\Delta_j' \subset H_{\alpha,t}$, then $t > n$ holds.
\end{itemize}
If there exist indices $k < j_1 < j_2 < \ldots j_l < s \leq p$, such that $\delta$ has a positive crossing or folding at each $j_1, \ldots, j_l$, and positive integers $p, q_1, \ldots, q_l$ such that
$$ \alpha_{i_s} = p \delta_{s-1} \cdots \delta_{k+1} \alpha_{i_k} + \sum_{h = 1}^l q_h \delta_{s-1} \cdots \delta_{j_h} \beta_{j_h}, $$
for $\beta_{j_h} = - \delta_{j_h} \alpha_{i_{j_h}}$. Then $\delta$ has a positive crossing at the position $s$.
\end{lemma}
\begin{proof}
Let $\alpha_{i_s} = p \delta_{s-1} \cdots \delta_{k+1} \alpha_{i_k} + \sum_{h=1}^l q_h \delta_{s-1} \cdots \delta_{j_h} \beta_{j_h}$,
and assume that $\delta$ has either a negative crossing or a positive folding.
With this assumption, Proposition \ref{gallery_origin} and applying $\delta_0 \cdots \delta_{j-1}$ yields the equation 
$$\gamma-m'\underline{\delta} = p (\alpha-n\underline{\delta}) + \sum_{h=1}^l q_h (-\gamma_h + m_h\underline{\delta}),$$
for some positive roots $\gamma_h$ and $\gamma$ and some integers $m_h$ and $m'$.

Since $\alpha$ is simple and the roots $\gamma$ and $\gamma_h$ are all positive, this equation can only hold if $\gamma_h$ and $\gamma$ are equal to $\alpha$, hence we arrive at
$$\alpha-m'\underline{\delta} = p (\alpha-n\underline{\delta}) + \sum_{h=1}^l q_h (-\alpha + m_h\underline{\delta}).$$

Dividing into its real and imaginary part we obtain
$$1 + \sum_{h=1}^l q_h = p  \text{ and } m' + \sum_{h=1}^l q_h m_h = p n.$$
Since all occurring roots are equal to $\alpha$, we know by the third requirement that $m_h > n$ for all $h$ and $m' > m$. Thus the two equalities contradict each other. 

Hence $\delta$ has to have a positive crossing at $s$.
\end{proof}


\begin{lemma}\label{new_terms_2}
Let $\kappa=\kappa^i$ for $i > 0$, $u = u_i$, $v = v_i$, and $t > v$. If there exist indices $u < j_1' < j_2' < \ldots j_{l'}' < v j_1 < j_2 < \ldots < j_l < t \leq p$, and positive integers $p$, $q_1, \ldots, q_l$, and $q_1', \ldots, q_{l'}'$ such that $\delta$ has 
\begin{itemize}
\item a positive crossing or folding at each $j_s'$, $1 \leq s \leq l'$,
\item a positive crossing or folding at each $j_s$ if $j_s \notin \IW_\alpha$,
\item a negative crossing or folding at each $j_s$ if $j_s \in \IW_\alpha$,
\end{itemize}
and at least one of the following two equalities holds
$$ \alpha_{i_t} = p \kappa_{t-1} \cdots \kappa_{v+1} \alpha_{i_v} + \sum_{h = 1}^l q_h \kappa_{t-1} \cdots \kappa_{j_h} \beta_{j_h} \text{ or }$$
$$ \alpha_{i_t} = p \kappa_{t-1} \cdots \kappa_{u+1} \alpha_{i_u} + \sum_{h = 1}^l q_h \kappa_{t-1} \cdots \kappa_{j_h} \beta_{j_h} + \sum_{h = 1}^{l'} q_h' \kappa_{t-1} \cdots \kappa_{j_h'} \beta_{j_h'}.$$
Then $\delta$ has a positive crossing at the position $t$ if $t \notin \IW_\alpha$ or a negative crossing if $t \in \IW_\alpha$.
\end{lemma}
\begin{proof}
Since the gallery $\kappa$ is negatively folded with respect to $\alpha$ after $v$ and positively folded for all other positive roots, we divide the set $\{1, \ldots, l\} = J_\alpha \cup \overline{J_\alpha}$, where $J_\alpha = \{ s \mid j_s \in \IW_\alpha \}$. Although the gallery is positively folded with respect to all roots before $v$, we do the same with $\{1, \ldots, l'\} = J_\alpha' \cup \overline{J_\alpha'}$ with $J_\alpha' = \{ s \mid j_s' \in \IW_\alpha \}$.

Let us start by assuming the first equality
$$ \alpha_{i_t} = p \kappa_{t-1} \cdots \kappa_{v+1} \alpha_{i_v} + \sum_{h \in J_\alpha} q_h \kappa_{t-1} \cdots \kappa_{j_h} \beta_{j_h} + \sum_{h \in \overline{J_\alpha}} q_h \kappa_{t-1} \cdots \kappa_{j_h} \beta_{h_j}, $$
with $\beta_{j_h}=-\kappa_{j_h}(\alpha_{i_{j_h}})$.  As in Lemma \ref{new_terms_1} we apply $\kappa_0 \cdots \kappa_{t-1}$, which yields
$$ \kappa_0 \cdots \kappa_{t-1} \alpha_{i_t} = p (-\alpha + n\underline{\delta}) + \sum_{h \in J_\alpha} q_h \left( \alpha - n_h\underline{\delta} \right) + \sum_{h \in \overline{J_\alpha}} q_h \left( -\gamma_h + n_h\underline{\delta} \right), $$
for some positive roots $\gamma_h \neq \alpha$ and integers $n_h$, with $n_h < n$ if $h \in J_\alpha$.

We need to dicsuss two possible cases:
\begin{enumerate}
\item \textbf{$t \in \IW_\alpha$:} Assume that $\kappa$ has either a negative folding or a positive crossing at $t$. In this case $\kappa_0 \cdots \kappa_{t-1} \alpha_{i_t} = -\alpha + n_t \underline{\delta}$, for some $n_t < n$. This implies $\overline{J_\alpha} = \emptyset$. Dividing our equation into its classical and imaginary part we obtain:
$$ p = 1 + \sum_{h \in J_\alpha} q_h \ \text{ and } \ pn = n_t + \sum_{h \in J_\alpha}q_hn_h.$$
Since $n_t < n$ and $n_h < n$ for all $h \in J_\alpha$, this is a contradiction.

\item \textbf{$t \notin \IW_\alpha$:} Assume that $\kappa$ has either a positive folding or a negative crossing at $t$. In this case $\kappa_0 \cdots \kappa_{t-1} \alpha_{i_t} = \beta - n_t \underline{\delta}$, for some $\beta \neq \alpha$ and $n_t \in \IZ$. We only look at the classical part of our equation and obtain:
$$ \beta + p\alpha + \sum_{h \in \overline{J_\alpha}} q_h \gamma_h = \sum_{h \in J_\alpha} q_h \alpha.$$
Since the root $\beta$ and all the $\gamma_h$ are positive roots different from $\alpha$ and $\alpha$ is simple, this equation cannot hold.
\end{enumerate}

\noindent We now proceed with the second equality
$$ \alpha_{i_t} = p \kappa_{t-1} \cdots \kappa_{u+1} \alpha_{i_u} + \sum_{h \in J_\alpha \cup \overline{J_\alpha}} q_h \kappa_{t-1} \cdots \kappa_{j_h} \beta_{j_h} + \sum_{h \in J_\alpha' \cup \overline{J_\alpha'}} q_h' \kappa_{t-1} \cdots \kappa_{j_h'} \beta_{j_h'}, $$
with $\beta_{j_h}=-\kappa_{j_h}(\alpha_{i_{j_h}})$ and $\beta_{j_h'}=-\kappa_{j_h'}(\alpha_{i_{j_h'}})$. As above we apply $\kappa_0 \cdots \kappa_{t-1}$, which yields
\begin{eqnarray*}
\kappa_0 \cdots \kappa_{t-1} \alpha_{i_t} = p (\alpha - n\underline{\delta}) &+& \sum_{h \in J_\alpha} q_h \left( \alpha - n_h\underline{\delta} \right) + \sum_{h \in \overline{J_\alpha}} q_h \left( -\gamma_h + n_h\underline{\delta} \right) \\
&+& \sum_{h \in J_\alpha'} q_h' \left( -\alpha + n_h' \underline{\delta} \right) + \sum_{h \in \overline{J_\alpha'}} q_h' \left( -\gamma_h' + n_h' \underline{\delta} \right),
\end{eqnarray*}
for some positive roots $\gamma_h \neq \alpha$, $\gamma_h' \neq \alpha$ and integers $n_h$, with $n_h < n$ if $h \in J_\alpha$ and $n_h'$ with $n_h' \geq n$ if $h \in J_\alpha'$.

Again we need to discuss the two cases:
\begin{enumerate}
\item \textbf{$t \in \IW_\alpha$:} Assume that $\kappa$ has either a negative folding or a positive crossing at $t$.
As before $\kappa_0 \cdots \kappa_{t-1} \alpha_{i_t} = -\alpha + n_t \underline{\delta}$, for some $n_t < n$. Again this implies $\overline{J_\alpha} = \overline{J_\alpha'} = \emptyset$. As before we divide into the classical and imaginary parts and obtain
$$ \sum_{h \in J_\alpha'} q_h' = p + 1 + \sum_{h \in J_\alpha} q_h \ \text{ and } \ \sum_{h \in J_\alpha'} n_h'q_h' = np + n_t + \sum_{h \in J_\alpha} n_hq_h.$$
Since $n_h' \geq n$, $n_t<n$, and $n_h < n$, these two equalities contradict each other.
\item \textbf{$t \notin \IW_\alpha$:} Assume that $\kappa$ has either a positive folding or a negative crossing at $t$. Applying $\kappa_0 \cdots \kappa_{t-1}$ to the left side yields $\beta - n_t \underline{\delta}$ for some $\beta \neq \alpha$ and $n_t \in \IZ$. We look at the classical part of this equation and order everything in a suitable way to obtain
$$ \beta + \sum_{h \in \overline{J_\alpha}} q_h \gamma_h + \sum_{h \in \overline{J_\alpha'}} q_h' \gamma_h' = \left( p + \sum_{h \in J_\alpha} q_h - \sum_{h \in J_\alpha'} q_h' \right) \alpha.$$
Since all the $\gamma_h$, $\gamma_h'$, and $\beta$ are positive roots different from $\alpha$, this equality cannot hold.
\end{enumerate}
\end{proof}
We commence with the changes of the gallery at a section $[u_i,v_i]$, but assume that $u_i \neq 0$. We will deal with the case $u_i = 0$ afterwards.

\begin{proposition}\label{general_2}
Fix $1 \leq i \leq r$ and assume $0 < u_i < p$. Let $g=g(a) \in \ku_{\kappa^{i-1}}$ be the coordinates of $g$ with respect to $\kappa^{i-1}$.
If for any $I \subset C([u_i,v_i])$ the inequality
$$ \sum_{n \in I} a_n \neq 0, $$
holds, then $g \in \ku_{\kappa^i}$.
\end{proposition}
\begin{proof}
For simplicity we will write $[u,v]$ instead of $[u_i,v_i]$. Let $n \in \IZ$ such that either $[u,v]$ is $\alpha$-stable at $n$ or $n = \boldsymbol{\mathfrak m}_\alpha^\delta$ if $u_i = \boldsymbol{\mathfrak t}_\alpha^\delta$. 
By the assumption $u < p$ we have to change $\kappa^i_u$ to ${\rm id}_W$.

Hence we look at the part
$$[x_{-\alpha_{i_u}}(a_u),g_{u+1}(a_{u+1}),\ldots,g_p(a_p)].$$
We use the Chevalley relations to change the gallery at the position $u$, to obtain
$$[x_{\alpha_{i_u}}(a_u^{-1})s_{\alpha_{i_u}}x_{\alpha_{i_u}}(a_u) a_u^{\alpha_{i_u}^\vee},g_{u+1}(a_{u+1}),\ldots,g_p(a_p)].$$
\noindent
We move the two terms $x_{\alpha_{i_u}}(a_u) a_u^{\alpha_{i_u}^\vee}$ to the right using the commutator formula. The term $a_u^{\alpha_{i_u}^\vee}$ poses no problem in this regard. The term $x_{\alpha_{i_u}}(a_u)$ produces new terms when moving to the right, which need to be controlled.
 
Let $u=k_1 < \ldots < k_s$ be such that $C([u,v])=\{k_i \mid 1 \leq i \leq s \}$. We first move $x_{\alpha_{i_u}}(a_u) a_{u}^{\alpha_{i_u}^\vee}$ to the right until it arrives in front of $g_{k_2}(a_{k_2})$. We also move all newly created terms to the right. 

If we look at the truncated gallery $[\kappa^{i-1}_0,\ldots, \kappa^{i-1}_{k_2-1}]$ and the index $k=k_1$ it satisfies the prerequisites of Lemma \ref{new_terms_1}, hence the new terms that are created in that part of the gallery will either add themselves to coefficients in front of positive crossings or move to the end and arrive in front of $g_{k_2}(a_{k_2})$.

Using the notation from Lemma \ref{new_terms_1} a new term that arrives in front of $g_{k_2}(a_{k_2})$ has a root of the form
$$p \delta_{k_2-1} \cdots \delta_{k+1} \alpha_{i_k} + \sum_{h = 1}^l q_h \delta_{k_2-1} \cdots \delta_{j_h} \beta_{j_h}.$$
Assuming that this is equal to $\alpha_{i_{k_2}}$, we apply $\delta_0 \cdots \delta_{k_2-1}$ to this equality and divide into the real and imaginary parts,
$$ (p-1)\alpha = \sum_{h=1}^l q_h \gamma_h \ \text{ and } \ (p-1)n = \sum_{h=1}^l q_h m_h. $$
Again it easily follows that all occurring $\gamma_h$ must be equal to $\alpha$. If we assume a single $m_h$ to be strictly greater than $n$ we have a contradiction. Thus the new terms can move past $g_{k_2}(a_{k_2})$ without changing the coefficient at that position.

The same argument also shows that new terms will move past all other critical indices. For this we only need to assume that at least one of the $m_h$ above is strictly greater than $n$. But this is no restriction, since if we assume that all $m_h$ are equal to $n$ then all the contributing indices $j_h$ correspond are critical indices and the commutator formula does not apply at these positions as the term $x_{\alpha_{i_u}}(a_u)$ arrives at such a position as $x_{\alpha_{i_{k_h}}}(a_u)$.

\textbf{\textit{Summary:}} Hence new terms either add themselves to coefficients at positive crossings or move past $k_s$.

The term $x_{\alpha_{i_u}}(a_u)$ can be moved to the right, since by Proposition \ref{gallery_origin} the only time where its root could be changed to a negative root are those in $C([u,v])$ and those changes are explicitely calculated in Lemma \ref{critical_calc}. Thus we can move it up to the end.

If we are in the case that $v = p$, we move $x_{\alpha_{i_u}}(a_u)$ to the end and it vanishes and the arguments above show that we can move all the new terms to the end of the gallery as well and they either vanish at the end or add themselves to terms at positive crossings.

Hence we assume that $v \neq p$, in this case we have a negative crossing at the position $v$. Hence by Lemma \ref{critical_calc} below, the terms $x_{\alpha_{i_u}}(a_u) a_{u}^{\alpha_{i_u}^\vee}$ will move to the right and arrive at the position $v$ as $x_{\alpha_{i_v}}(A) (-A)^{\alpha_{i_v}^\vee}$.

We do an inverse transformation to the one at position $u$
$$ x_{\alpha_{i_v}}( A ) (-A)^{\alpha_{i_v}^\vee} s_{i_v} = x_{-\alpha_{i_v}}( A^{-1} ) x_{\alpha_{i_v}}( -A ). $$

Thus like in the previous situation we have a term of the form $x_{\alpha_{i_v}}( -A )$ that needs to move through the rest of the gallery. Since, in this case, all hyperplanes corresponding to $\alpha$, occurring after the position $v$ are of a height strictly smaller than $n$ the term itself will move through the whole gallery and we only need to take care of the terms created via the Chevalley formula. The behaviour of the new terms is controlled by applying Lemma \ref{new_terms_2}.
\end{proof}


\begin{lemma} \label{critical_calc}
We use the same notations as in the proof of Proposition \ref{general_2}, with $g=g(b)$ being the coordinates with respect to $\kappa^i$, then the following holds
$$ b_u = a_u^{-1} \ \text{ and } \ b_{k_j} = \frac{a_{k_j}}{\left( \sum_{l=1}^{j-1} a_{k_l} \right) \left( \sum_{l=1}^j a_{k_l} \right)}.$$
\end{lemma}
\begin{proof}
The equality $ b_u = a_u^{-1}$ is obvious by the proof of Proposition \ref{general_2}. As proved in Proposition \ref{general_2}, when moving $x_{\alpha_{i_u}}(a_u) a_u^{\alpha_{i_u}^\vee}$ to the right the terms that are created through the Chevalley commutator formula will play no role when looking at the critical indices. By Proposition \ref{gallery_origin}, we know that $\delta_{k_2-1} \cdots \delta_{k_1+1} \alpha_{k_1} = \alpha_{k_2}$ and thus
\begin{eqnarray*}
& &x_{\alpha_{i_{k_2}}}(a_{k_1})(-a_{k_1})^{\alpha_{i_{k_2}}^\vee} x_{-\alpha_{i_{k_2}}}(a_{k_2})\\
&=& x_{-\alpha_{i_{k_2}}} \left(a_{k_1}^{-1}a_{k_2}(a_{k_1}+a_{k_2})^{-1} \right) x_{\alpha_{i_{k_2}}}(a_{k_1}+a_{k_2}) (-(a_{k_1}+a_{k_2}))^{\alpha_{i_{k_2}}^\vee}.
\end{eqnarray*}
We repeat the process with the last two terms. Hence we have the same calculation with $a_{k_1}+a_{k_2}$ instead of $a_{k_1}$ for the next critical index. Thus, inductively, we obtain
$$b_{k_j} = \frac{a_{k_j}}{\left( \sum_{l=1}^{j-1} a_{k_l} \right) \left( \sum_{l=1}^j a_{k_l} \right)}. $$
\end{proof}


As mentioned before we need to make modifications for the case that the first coefficient of the gallery is involved in the calculations.

\begin{proposition} \label{general_1_alt}
We assume that ${\boldsymbol{\mathfrak j}_\alpha^\delta}=0$. If 
$$ g_0(\underline{a_\beta}) = \prod_{\beta < 0, \delta_0^{-1}(\beta)<0}^{\leftarrow} x_\beta(a_\beta) \cdot \delta_0, $$
with $a_{\beta} \neq 0$ for all $\beta$, then $g \in \ku_{\kappa^0}^{s_\alpha}$.
\end{proposition}
\begin{proof}
We start by moving the coefficient $x_{-\alpha}(a_{-\alpha})$ to the right such that it is located in front of $\delta_0$. Since, for positive integers $p$ and $q$ and a negative root $\beta$, $-p\alpha + q \beta$ is of height strictly smaller than both $-\alpha$ and $\beta$ and the set of roots in occurring in our product is closed under summation, we know that any new term that is created while moving $x_{-\alpha}(a_{-\alpha})$ to the right will add itself to one of the coefficients.

Thus we arrive at the following form
$$ g_0(\underline{a_\beta}) = \prod_{\beta < 0, \beta \neq -\alpha, \delta_0^{-1}(\beta)<0}^{\leftarrow} x_\beta(a_\beta') x_{-\alpha}(a_{-\alpha}) \cdot \delta_0, $$
where $a_\beta' = a_\beta + f(a_\gamma \mid ht(-\alpha) \geq ht(\gamma) > ht(\beta))$ and $f$ is a polynomial. This triangular pattern of coordinate changes will be important later on.

Next we proceed with our usual transformation of $x_{-\alpha}(a_{-\alpha})$. With our assumptions we can write $g_0(\underline{a_\beta})$ as
$$ g_0(\underline{a_\beta}) = \prod_{\beta < 0, \beta \neq -\alpha, \delta_0^{-1}(\beta)<0}^{\leftarrow} x_\beta(a_\beta') x_{\alpha}(a_{-\alpha}^{-1})s_\alpha x_\alpha(a_{-\alpha}) a_{-\alpha}^{\alpha^\vee} \cdot \delta_0. $$
Since $\delta_0^{-1}\alpha > 0$, we can move $x_\alpha(a_{-\alpha}) a_{-\alpha}^{\alpha^\vee}$ past $\delta_0$. The changes that occur afterwards when this term moves through the gallery are exactly the same as before and are dealt with by using Lemma \ref{new_terms_1}.

Hence we are left with 
$$ g_0(\underline{a_\beta}) = \prod_{\beta < 0, \beta \neq -\alpha, \delta_0^{-1}(\beta)<0}^{\leftarrow} x_\beta(a_\beta') x_{\alpha}(a_{-\alpha}^{-1}) \cdot s_\alpha\delta_0. $$
\noindent
Looking at the definition of $\ku_{\kappa^0}^{s_\alpha}$, we see that any $\gamma$ appearing in the leading coefficient needs to satisfy the following:
$$ \gamma <_{s_\alpha} 0 :\Leftrightarrow s_\alpha \gamma < 0 \ \text{ and } \ (s_\alpha \delta_0)^{-1}\gamma < 0. $$
This must be checked for the roots that appear.

The first of these conditions is obviously satisfied by all roots that we have. The second condition is not always satisfied and we need to get rid of those roots that don't satisfy it. Hence we divide our set of roots that appear into two disjoint subsets
$$R_\alpha = \{ \beta < 0 \mid \beta \neq -\alpha, \delta_0^{-1}\beta < 0, (s_\alpha\delta_0)^{-1}\beta < 0 \},$$
which are the ones that satisfy the second condition and
$$R_\alpha' = \{ \beta < 0 \mid \delta_0^{-1}\beta < 0, (s_\alpha\delta_0)^{-1}\beta > 0 \},$$
the ones that do not.

In Proposition \ref{origin_parameter_change} we take a closer look at the roots which we have to substitute for the ones in $R_\alpha'$. But for this proof we will just eliminate the coefficients that we don't need. Hence we will just move all elements corresponding to roots in $R_\alpha'$ past $x_\alpha(a_\alpha'')$ and $s_\alpha \delta_0$, starting with the one furthest to the left, hence the highest. Again we produce new terms, but as all of these correspond to negative roots, different from $-\alpha$, they satisfy the first condition above. If they don't satisfy the second condition above, we move them past $s_\alpha \delta_0$ as well, and leave them where they are if they do satisfy the condition.

As usual we have to look at how the new terms affect our gallery. In this case we cannot directly apply lemmas \ref{new_terms_1} or \ref{new_terms_2}, since we have not produced these new terms via Chevalley relations. Nevertheless the calculations are very similar and we will just outline them briefly.
Let $x_{\delta_0^{-1}s_\alpha\gamma}(c)$ be such that $\gamma < 0$ and $(s_\alpha \delta_0)^{-1}\gamma > 0$ with $c \in \IC$. With the same notations as in Lemma \ref{new_terms_1} we arrive at an equation of the form
$$\alpha_{i_s} = p \delta_{s-1} \cdots \delta_{1} \delta_0^{-1} s_\alpha\gamma + \sum_{h=1}^l q_h \delta_{s-1} \cdots \delta_{j_h} \beta_{j_h},$$
with $\beta_{j_h}=-\delta_{j_h}(\alpha_{i_{j_h}})$.
Again we apply $\delta_0 \cdots \delta_{j-1}$ and assume that the gallery has a negative crossing at $s$ or a positive folding. The classical part of the equation is then
$$\beta + \sum_{h=1}^l q_h \gamma_h = p s_\alpha\gamma, $$
for some $\beta > 0$ and $\gamma_h > 0$. Which is a contradiction as the left side is a positive linear combination of positive roots, while the right side is a positive multiple of a negative root. Hence the new terms can also only contribute to the coefficients at positive crossings. This completes the proof.
\end{proof}
\begin{proposition}\label{general_2_alt}
Assume $u_r = 0$. Let $g=g(a) \in \ku_{\kappa^{r-1}}$.
Then $g \in \ku_{\kappa^r}^{s_\alpha}$, if for any $I \subset C([u_r,v_r])$ the following inequality holds:
$$ \sum_{n \in I} a_n \neq 0. $$
\end{proposition}
\begin{proof}
The proof is a combination of the proofs of Proposition \ref{general_1_alt} and Proposition \ref{general_2}. The way $\kappa^{r-1}_0$ is changed is exactly the same as in Proposition \ref{general_1_alt}, while the rest of the calculations that the occurring terms can be moved to the right in the gallery and do not change the structure of the gallery is exactly as in Proposition \ref{general_2} and again uses lemmas \ref{new_terms_1} and \ref{new_terms_2}.
\end{proof}
Combining the propositions we obtain at the following statement.

\begin{theorem} \label{case_simple_refl}
Let $\delta \in \Gamma_{LS}^+(\gamma_\lambda)$. There exists a dense open subset $O_{id} \subset C(\delta)$ such that 
$$O_{id} \subset \bigcap_{\alpha \text{ simple roots}} \ku_{\Xi_{s_\alpha}(\delta)}^{s_\alpha}.$$
\end{theorem}
\begin{proof}
This follows from the change of coordinates that occur during the proofs of propositions \ref{general_2}, \ref{general_1_alt}, and \ref{general_2_alt}. When we change the coordinates at the position ${\boldsymbol{\mathfrak j}_\alpha^\delta}$ or $u_i$ by applying one of these propositions, we only change those parameters that occur afterwards and never before.

Hence we can start with the transformation at ${\boldsymbol{\mathfrak j}_\alpha^\delta}$ and then proceed with the transformations at the stable sections $[u_i,v_i]$ in increasing order. This yields a finite set of inequalities for our parameters, with each proposition giving us inequalities for different, pairwise non-intersecting, sets of parameters. Thus these inequalities can never contradict themselves and we obtain a dense subset where they are all satisfied.
\end{proof}
\subsection*{Independence of parameters} \label{parameter_change}

By applying propositions \ref{general_2}, \ref{general_1_alt}, and \ref{general_2_alt} we have seen that for a suitably general gallery $g$ in $C(\delta)$ we can rewrite the coordinates with respect to $\Xi_{s_\alpha}(\delta)$ and can easily see that $r_{w_0s_\alpha}(g)=\Xi_{s_\alpha}(\delta)$. To use this as the basis for an inductive argument, we need that such a gallery $g$ can also be chosen suitably general enough to apply the coordinate changes from $\Xi_{s_\alpha}(\delta)$ to $\Xi_{s_\alpha s_\beta}(\delta)$ for any simple root $\beta \neq \alpha$.

By applying either Proposition \ref{general_1_alt} if needed and afterwards propositions \ref{general_2} and \ref{general_2_alt} we obtain the followings.

\begin{proposition} \label{independence_non_origin}
For $g = g(a) \in O_{id}$ and $g = g(b) \in \ku_{\Xi_{s_\alpha}(\delta)}^{s_\alpha}$. We have relations for $j > 1$:\\
\textbf{$\Xi_{s_\alpha}(\delta)$ has a positive crossing at $j$:} 
\begin{enumerate}
\item If $j = u_i$ for $1 \leq i \leq r$: $b_j = f_j(a_s \mid s \in I^c_{<j}) a_j^{-1} + G_j(a_k \mid k<j)$
\item Otherwise: $b_j = f_j(a_s \mid s \in I^c_{<j}) a_j + G_j(a_k \mid k<j)$
\end{enumerate}
\textbf{$\Xi_{s_\alpha}(\delta)$ has a negative crossing at $j$:} $b_j = 0$.\\
\textbf{$\Xi_{s_\alpha}(\delta)$ has a positive folding at $j$:}
\begin{enumerate}
\item If $j \in C([u_i,v_i])$ for some $1 \leq i \leq r$: 
$$ b_j = f_j(a_s \mid s \in I^c_{<j}) a_j \left( \sum_{s \in C([u_i,v_i]), s < j} a_s \right)^{-1}  \left( \sum_{s \in C([u_i,v_i]), s \leq j} a_s \right)^{-1}.$$
\item If $j = v_i$ for some $i > 1$: $ b_j = f(a_s \mid s \in I^c_{<j}) \left( \sum_{s \in C([u_i,v_i])} a_{s} \right)^{-1}.$
\item Otherwise: $b_j = f(a_s \mid s \in I^c_{<j}) a_j.$
\end{enumerate}
Here $f_j$ denotes a non-zero rational function and $G_j$ a rational function.
\end{proposition}
\begin{proof}
This follows immediately from the proofs of propositions \ref{general_2}, \ref{general_1_alt}, and \ref{general_2_alt}.
\end{proof}

\begin{corollary}
Let $g \in O_{id}$, then $r_{w_0s_\alpha}(g) = \Xi_{s_\alpha}(\delta)$ for all simple roots $\alpha$.
\end{corollary}
\begin{proof}
Let $\alpha$ be a simple root and $g = g(b) \in \ku_{\Xi_{s_\alpha}(\delta)}^{s_\alpha}$. Combining Proposition \ref{independence_non_origin} and Theorem \ref{BB_cell} we have $g \in r_{w_0s_\alpha}^{-1}(\Xi_{s_\alpha}(\delta))$.
\end{proof}

We also need to look at the parameters at the first position of our gallery and have seen in propositions \ref{general_1_alt} that the changes are more complicated and can not be controlled in such detail as the other ones.

We want to make a short observation on the question for which roots we have to exchange the ones in $R_\alpha'$. 

\begin{lemma} \label{roots_first_term}
Suppose $\delta_0^{-1} \alpha > 0$. For  $R'=\{ \beta \in \Phi \mid \beta <_{s_\alpha} 0, (s_\alpha \delta_0)^{-1}\beta < 0 \}$ it holds
$$R' = R_\alpha \cup \{\alpha \} \cup s_\alpha R_\alpha',$$
with $R_\alpha$ and $R_\alpha'$ as in the proof of Proposition \ref{general_1_alt}.
\end{lemma}
\begin{proof}
The right side is a subset of the left side by definition and a short moment of thought shows that both sides have the same cardinality.
\end{proof}

To control the changes of parameters at the first coefficient of our gallery we make the following simple geometric observation.

\begin{proposition} \label{origin_parameter_change}
Suppose $\delta_0^{-1} \alpha > 0$. Let $U_e = U^- \delta_0 B$ and $U_\alpha = (U^-)^{s_\alpha} s_\alpha \delta_0 B$. Then $\dimension (U_e) = \dimension (U_\alpha)$ and $U_e \cap U_\alpha$ is dense in $U_e$ and in $U_\alpha$. 
\end{proposition}
\begin{proof}
By definition of our gallery and by Lemma \ref{roots_first_term}, we know
$$U_e = \prod_{\beta \in R_\alpha} \ku_\beta \prod_{\beta \in R_\alpha'} \ku_\beta \cdot \ku_{-\alpha} \cdot \delta_0 B$$
and
$$U_\alpha = \prod_{\beta \in R_\alpha} \ku_\beta \prod_{\beta \in s_\alpha R_\alpha'} \ku_{\beta} \cdot \ku_{\alpha} \cdot s_\alpha\delta_0 B.$$
This already implies that $\dimension (U_e) = \dimension (U_\alpha)$. A generic elemtent in $U_e$ is of the form
$$ g = \prod_{\beta \in R_\alpha} x_\beta(a_\beta) \prod_{\beta \in R_\alpha'} x_\beta(a_\beta) \cdot x_{-\alpha}(a_\alpha) \cdot \delta_0 B,$$ for some complex parameters $a_\beta$.
If we assume that $a_\alpha \neq 0$ we have already seen in Proposition \ref{general_1_alt} that we can do a simple transformation and obtain
$$ g = \prod_{\beta \in R_\alpha} x_\beta(a_\beta) \prod_{\beta \in R_\alpha'} x_\beta(a_\beta) \cdot x_{\alpha}(a_\alpha^{-1}) \cdot s_\alpha \delta_0 B.$$
Both products consist of elements in $(U^-)^{s_\alpha}$, thus we have $g \in U_\alpha$. By dimension reasons we know that the set of those elements in $U_e$, such that $a_\alpha \neq 0$, is dense in both $U_e$ and $U_\alpha$. Hence the intersection is dense a fortiori.
\end{proof}

With a bit more calculations and controlling of the parameters this can be made a bit more precise.

\begin{remark} \label{origin_improved}
Looking at the final way the element $g$ was written in the proof of Proposition \ref{origin_parameter_change}, one can calculate in a straight-forward but tedious manner that the existence of a generic parameter for $x_\beta(a_\beta)$, with $\beta \in R_\alpha'$ implies that the parameter for $x_{s_\alpha \beta}(c_\beta)$ is generic if we write $g$ in the corresponding coordinates for $U_\alpha$ as mentioned at the beginning of the proof. 
\end{remark}

\subsection*{General case}

We now want to use the result for the simple roots to obtain the general result by an inductive argument. What we have seen so far is, if we start with $\delta \in \Gamma_{LS}^+(\gamma_\lambda)$, we know that there exists a dense subset $O_{id} \subset C(\delta)$ such that $r_{w_0s_\alpha}(g)=\Xi_{s_\alpha}(\delta)$ for all simple roots $\alpha$ and $g \in O_{id}$. 
If we now fix a simple root $\beta$, the same is also true for $\Xi_{s_\beta}(\delta) \in \Gamma_{LS}^{s_\beta}(\gamma_\lambda)$, hence there exists a dense subset $O_{s_\beta} \subset C_{s_\beta}(\Xi_{s_\beta}(\delta))$ such that $r_{w_0s_\beta s_\alpha}(g)=\Xi_{s_\beta s_\alpha}(\delta)$ for all simple roots $\alpha$ and $g \in O_{s_\beta}$.

Thus to obtain the result for the general case we have to show that $O_{id} \cap O_{s_\beta}$ is dense in $O_{id}$ and also more generally if we take any $w  \in W \setminus \{w_0\}$ that $O_{id} \cap O_w$ is dense in $O_{id}$, where $O_w$ is defined in the same way as $O_{s_\beta}$ above.

\begin{theorem}\label{finalresult}
For $\delta \in \Gamma_{LS}^+(\gamma_\lambda)$, there exists a dense subset $O_\delta \subset C(\delta)$ such that for every $g \in O_\delta$ it holds
$$ r_{w_0w}(g) = \Xi_w(\delta) \text{ for all } w \in W. $$
\end{theorem}
\begin{proof}
We want to define the subset $O_\delta$ inductively. For this let
$$ W_1 \subset W_2 \subset \ldots \subset W_l = W, $$ be the subsets of $W$ such that $W_i = \{ w \in W \mid l(w) \leq i \}$ and $l=l(w_0)$. We want to define a decreasing series of dense subsets $O_\delta^i$ of $C(\delta)$ such that for any $g \in O_\delta^i$ and $w \in W_i$ it holds that 
$$ r_{w_0w}(g)=\Xi_w(\delta). $$
Of course the subset $O_\delta^l$ is the one needed for the proof of the statement.

Theorem \ref{case_simple_refl} is already the proof for $W_1$, but to use an inductive argument we have to be a bit more careful about the parameters to make it work. As seen before in this section in Proposition \ref{independence_non_origin}, we do not explicitly know how the parameters change, but the structure can be seen quite well. One observation is that the parameters for all indices of positive foldings stay algebraically independent if we view them as rational functions with our original parameters as indeterminantes.

Hence let us take $w \in W_i$ and assume that the theorem already holds for $W_{i-1}$. Let us now take $w' \in W_{i-1}$ and $w'\alpha \in w'\Phi^+$ a simple root such that $w = w's_\alpha = w' s_\alpha w'^{-1} w'$. For $g$ in $O_\delta^{i-1}$ we know that we can successively change the coordinates until we have written $g$ with respect to $\Xi_{w'}(\delta)$. 

If we now want to apply the coordinate change to obtain $g$ with respect to the combinatorial gallery $\Xi_w(\delta)$, we have to make sure that we can choose generic parameters for $g$ in such a way that they fulfil all the needed inequalities. By Proposition \ref{independence_non_origin} we know that by starting with generic coefficients, we will always have generic coefficients at all positive foldings, the only dependencies occur with some positive crossings that are of no concern for the calculations. The only position where we have to make sure that we are not missing any coefficients, is $g_0$.

For this let us take $w=s_{i_1} \cdots s_{i_m}$ a reduced decomposition of $w$, such that $s_{i_m}=s_\alpha$. Hence we need to successively transform the gallery to $\Xi_{w^l}(\delta)$ for $0 \leq l \leq m$. For ease of notation we write $w^l=s_{i_1} \cdots s_{i_l}$ for $1 \leq l \leq m$, and define $\beta_l=w^{l-1}\alpha_{i_l}$. We assume that we could already write our gallery with respect to $w^{m-1}$, hence we now need to apply the changes with respect to $\beta_m$. If propositions \ref{general_1_alt} and \ref{general_2_alt} do not apply in this situation it just means that we do not need to change anything at the first position of our gallery, thus no problem can arise. Hence let us assume that one of the propositions applies.

We need to fix some integers $0=k_1 < \ldots < k_r=m$ in $\{0, \ldots, m\}$. These shall be all the positions where we needed to change the first position of our gallery when moving from the coordinates with respect to $\Xi_{w^{k_l}}(\delta)$ to the ones with respect to $\Xi_{w^{k_l+1}}(\delta)$. We also fix the following elements
$$t_l = w^{k_l-1} s_{i_{k_l}} (w^{k_l-1})^{-1}, $$
these are exactly the reflections that are added to our element of the Weyl group at the first position in each of the steps where we apply one of the two proposition. Hence the Weyl group element at the first position is of the form.
$$ t_{r-1} \cdots t_1 \delta_0.$$
In addition we know that 
$$ (t_{r-1} \cdots t_1 \delta_0)^{-1} (-\beta_{k_r}) < 0, $$
as we otherwise would not need to apply one of the propositions. The question is if the one-parameter subgroup corresponding to $-\beta_{k_r}$, has a generic parameter in this step. To be sure that this does not pose a problem, we need to use Proposition \ref{origin_parameter_change}, Remark \ref{origin_improved}, and the following simple calculation to see that the one-parameter subgroup really has a generic parameter.

For this let us define $\gamma_r=\beta_{k_r}$. Then we know
$$(t_{r-1} \cdots t_1 \delta_0)^{-1}(-\gamma_r) <0,$$
hence
$$(t_{r-2} \cdots t_1 \delta_0)^{-1}(-t_{r-1}\gamma_r) < 0.$$
In addition its a straightforward calculation to see that 
$$ (w^{k_{r-1}-1})^{-1} (-t_{r-1} \gamma_r) <0 \text{ and } (w^{k_{r-1}-1})^{-1} (-\gamma_r) <0, $$
just by definition and the fact that $w=s_{i_1} \cdots s_{i_m}$ is a reduced expression.

If we now look at $(t_{r-2} \cdots t_1 \delta_0)^{-1}(-\gamma_r)$, there are two possibilities. Either this is negative, then by definition
$$ -\gamma_r \in R_{\beta_{k_{r-1}}}$$
and we define $\gamma_{r-1} = \gamma_r$. On the other hand, if $(t_{r-2} \cdots t_1 \delta_0)^{-1}(-\gamma_r)$ is positive, we know
$$ -t_{r-1}\gamma_r \in R_{\beta_{k_{r-1}}}'$$
by definition and we fix $\gamma_{r-1} = t_{r-1}\gamma_r$.

No matter how we defined $\gamma_{r-1}$ in each case, the one parameter subgroup corresponding to $-\beta_{k_r}$ exists in our situation if the one for $-\gamma_{r-1}$ existed in the previous transformation step $k_{r-1}$. Thus we iterate this process. 

If we inductively defined $\gamma_{r-j}$ such that 
$$(t_{r-j-1} \cdots t_1 \delta_0)^{-1} (-\gamma_{r-j})<0 \text{, } (w^{k_{r-j}-1})^{-1} (-\gamma_{r-j})<0 \text{, and } $$
$$ (w^{k_{r-j}-1})^{-1} (-t_{r-j}\gamma_{r-j}) < 0,$$
i.e., $\gamma_{r-j}$ is one of the roots whose one-parameter subgroup is allowed in the transformation step $k_{r-j}$.
Then we can again look at 
$$(t_{r-j-2} \cdots t_1 \delta_0)^{-1} (-\gamma_{r-j})$$
and if this is negative we define $\gamma_{r-j-1}$ as $\gamma_{r-j}$ and otherwise as $t_{r-j-1} \gamma_{r-j}$. The properties of $\gamma_{r-j}$ above imply 
$$(w^{k_{r-j-1}-1})^{-1} (-\gamma_{r-j-1})<0 \text{ and } (w^{k_{r-j-1}-1})^{-1} (-t_{r-j-1}\gamma_{r-j-1}) < 0.$$
Thus again we have an element that is either in $R_{\beta_{k_{r-j-1}}}$ or $R_{\beta_{k_{r-j-1}}}'$.

	After we have inductively defined all these elements we look at $-\gamma_0$, which is an element that satisfies $-\gamma_0<0$ and $\delta_0^{-1} (-\gamma_0)<0$. Thus the one-parameter subgroup corresponding to $-\gamma_0$ must exist by assumption on the genericness of the initial parameters. By definition it will be transformed to a element in the one-parameter subgroup corresponding to $-\beta_{k_r}$, after the transformations to $\Xi_{w^{m-1}}(\delta)$. Furthermore Proposition \ref{origin_parameter_change} and Remark \ref{origin_improved} tell us that this subgroup will have a generic parameter.

Thus we can use the proposition and have a generic non-zero parameter to work with. This means that we can perform the transformation for $w$ and we can thus define $O_\delta^i$ by looking at all elements $w$ of length $i$. This completes the proof.
\end{proof}

\subsection*{Results}

Combining Theorem \ref{finalresult} with Proposition \ref{retraction_galleries} we obtain the following results about MV-cycles and MV-polytopes.

\begin{corollary}
Let $M_\delta$ be an MV-cycle corresponding to $\delta \in \Gamma_{LS}^+(\gamma_\lambda)$, then 
$$ M_\delta := \overline{\bigcap_{w \in W} S^w_{\nu_w}}, \text{ with } \nu_w:=wt(\Xi_w(\delta)). $$
\end{corollary}

\begin{corollary}
Let $M_\delta$ be as above and $P=\Phi(M_\delta)$ the corresponding MV-polytope, then
$$ P = P_\delta:=\conv(\{wt(\Xi_w(\delta)) \mid w \in W \} ). $$
\end{corollary}

\begin{corollary}
Let $\delta \in \Gamma_{LS}^+(\gamma_\lambda)$, then
$$ \bigcap_{w \in W} C_w(\Xi_w(\delta)) $$
is dense in $C(\delta)$.
\end{corollary}

\bibliographystyle{alpha}
\bibliography{literature}

\end{document}